\documentclass[11pt]{amsart}
\usepackage{latexsym,amscd,amssymb, graphicx, color, amsthm, bm, amsmath, cancel, enumitem, ytableau}  
\usepackage{tikz}
\usepackage[margin=1in]{geometry}
\usepackage{hyperref}
\usepackage[all,cmtip]{xy}
\usepackage{mathtools}
\usepackage{rotating}
\usetikzlibrary{math, arrows.meta}
\usetikzlibrary{calc}
\usepackage[backend=bibtex]{biblatex}
\numberwithin{equation}{section}

\newtheorem{theorem}{Theorem}[section]
\newtheorem{proposition}[theorem]{Proposition}
\newtheorem{corollary}[theorem]{Corollary}
\newtheorem{lemma}[theorem]{Lemma}
\newtheorem{conjecture}[theorem]{Conjecture}

\newtheorem{problem}[theorem]{Problem}

\newtheorem{remark}[theorem]{Remark}
\newtheorem{definition}[theorem]{Definition}

\theoremstyle{definition}

\newcommand{\Hilb}{{\mathrm{Hilb}}}

\newcommand{\gl}{\mathrm{GL}}
\newcommand{\Frob}{{\mathrm{Frob}}}
\newcommand{\symm}{{\mathfrak{S}}}
\newcommand{\II}{{\mathbf{I}}}
\newcommand{\gr}{{\mathrm {gr}}}

\newcommand{\Res}{\mathrm{Res}}

\newcommand{\grFrob}{{\mathrm{grFrob}}}

\newcommand{\ZZZ}{{\mathcal{Z}}}
\newcommand{\UZ}{{\mathcal{UZ}}}
\newcommand{\SSS}{{\mathcal{S}}}

\newcommand{\MMM}{{\mathcal{M}}}

\newcommand{\RRR}{{\mathcal{R}}}

\newcommand{\CC}{{\mathbb{C}}}

\newcommand{\ZZ}{{\mathbb{Z}}}

\newcommand{\Mat}{{\mathrm{Mat}}}

\newcommand{\SYT}{\mathrm{SYT}}

\newcommand{\spa}{\mathrm{span}}

\newcommand{\xxx}{{\mathbf{x}}}

\newcommand{\mmm}{{\mathfrak{m}}}

\newcommand{\es}{{\mathfrak{es}}}
\newcommand{\ex}{{\mathcal{EX}}}
\newcommand{\eshadow}{{\mathcal{ES}}}

\newcommand{\diag}{<_{\text{Diag}}}
\newcommand{\eqdiag}{\le_{\text{Diag}}}
\newcommand{\ind}{\mathrm{Ind}}

\title{An extension of Viennot's shadow to rook placements via orbit harmonics}
\author{Jasper M. Liu}
\author{Hai Zhu}
\address{Department of Mathematics, UC San Diego, La Jolla, CA, 92093, USA}
\email{(mol008, haz138)@ucsd.edu}
\date{\today}

\begin{document}

\begin{abstract}
    For fixed positive integers $n,m,r$, let $\Mat_{n \times m}(\CC)$ be the affine space of $n \times m$ complex matrices with coordinate ring $\CC[\xxx_{n \times m}]$. We define a homogeneous ideal $I_{n,m,r}$, where the graded quotient $\CC[\xxx_{n \times m}]/I_{n,m,r}$ is obtained from the orbit harmonics deformation of the matrix loci corresponding to all rook placements of size at least $r$. By extending rook placements to elements in $\symm_{n+m-r}$ and applying Viennot's shadow line avatar of the Schensted correspondence, we compute the standard monomial basis of the quotient $\CC[\xxx_{n \times m}]/I_{n,m,r}$ with respect to diagonal monomial orders. We also determine the graded $\symm_n\times\symm_m$-module structure of $\CC[\xxx_{n \times m}]/I_{n,m,r}$.
\end{abstract}

\maketitle

\section{Introduction}\label{sec:intro}
Let $\xxx_N = (x_1,x_2,\dots,x_N)$ be a collection of $N$ variables and let $\CC[\xxx_N] := \CC[x_1,x_2,\dots,x_N]$ be the polynomial ring over these variables. For a finite locus of points $\ZZZ \subseteq \CC^N$, its vanishing ideal is
\begin{equation}
    \II(\ZZZ) :=  \{f \in \CC[\xxx_N]: f(z) = 0 \text{ for all } z \in \ZZZ \}.
\end{equation}
The coordinate ring $\CC[\ZZZ]$ can then be identified with $\CC[\xxx_N]/\II(\ZZZ)$. The method of {\em orbit harmonics} replaces the ideal $\II(\ZZZ)$ with its associated graded ideal $\gr \, \II(\ZZZ)$, extending this identification of (ungraded) vector spaces to
\begin{equation}
    \CC[\ZZZ] \cong \CC[\xxx_N]/\II(\ZZZ) \cong \CC[\xxx_N]/\gr \, \II(\ZZZ) =: R(\ZZZ).
\end{equation}

If $\ZZZ$ is invariant under the action of a finite matrix group $G \subseteq \gl_N(\CC)$, this isomorphism can be treated as an isomorphism of $G$-modules, and $R(\ZZZ)$ has the additional structure as a graded $G$-module.

Geometrically, the method of orbit harmonics linearly deforms the locus $\ZZZ$ to a scheme of multiplicity $|\ZZZ|$ supported at the origin. In the picture below, this deformation is drawn schematically for $\ZZZ$, where $\ZZZ$ is a locus of six points stable under the action of $\symm_3$ generated by reflections across the three lines.

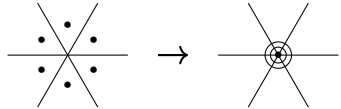
\begin{center}
 \begin{tikzpicture}[scale = 0.2]
\draw (-4,0) -- (4,0);
\draw (-2,-3.46) -- (2,3.46);
\draw (-2,3.46) -- (2,-3.46);

 \fontsize{5pt}{5pt} \selectfont
\node at (0,2) {$\bullet$};
\node at (0,-2) {$\bullet$};

\node at (-1.73,1) {$\bullet$};
\node at (-1.73,-1) {$\bullet$};
\node at (1.73,-1) {$\bullet$};
\node at (1.73,1) {$\bullet$};

\draw[thick, ->] (6,0) -- (8,0);

\draw (10,0) -- (18,0);
\draw (12,-3.46) -- (16,3.46);
\draw (12,3.46) -- (16,-3.46);

\draw (14,0) circle (15pt);
\draw(14,0) circle (25pt);
\node at (14,0) {$\bullet$};

 \end{tikzpicture}
\end{center}

The method of orbit harmonics dates back to Kostant ~\cite{Kostant1963LieGR}. It has seen its application in numerous settings, including cohomology theory ~\cite{GARSIA199282}, Macdonald theory ~\cite{griffin2021orderedsetpartitionsgarsiaprocesi, HAGLUND2018851}, cyclic sieving \cite{Oh-Rhoades}, Donaldson–Thomas invariants ~\cite{Reineke-Rhoades}, and Ehrhart theory ~\cite{reiner2024harmonicsgradedehrharttheory}. In cases where the locus $\ZZZ$ has nice combinatorial structures or symmetries, one can usually expect that the algebraic properties of $R(\ZZZ)$ are governed by the combinatorial properties of $\ZZZ$.

Let $\xxx_{n \times m} = (x_{i,j})_{1 \leq i \leq n , \, 1 \leq j \leq m}$ be an $n \times m$ matrix of variables. Let $\CC[\xxx_{n \times m}]$ be the polynomial ring over these variables and $\Mat_{n\times m} (\CC)$ be the affine space of $n \times m$ complex matrices. The application of orbit harmonics to finite matrix loci $\ZZZ \subseteq \Mat_{n \times m}(\CC)$ was initiated by Rhoades \cite{rhoades2024increasing}. He considered the case $n=m$ with the point locus $\ZZZ = \symm_n \subseteq \Mat_{n \times n}(\CC)$ of $n \times n$ permutation matrices. This permutation matrix locus carries an action of $\symm_n \times \symm_n$ by left and right multiplication. Algebraic properties of $R(\symm_n)$ are governed by the longest increasing subsequences in $\symm_n$ and Viennot's shadow line construction. Liu \cite{Liu} extended this work to the locus $\ZZZ = \ZZ_r \wr \symm_n$ of $r$-colored permutation matrices. Liu, Ma, Rhoades, Zhu ~\cite{Liu-Ma-Rhoades-Zhu} further studied the matrix loci $\ZZZ = \{w \in \symm_n : w^2 = 1 \} \subseteq \Mat_{n \times n} (\CC)$ of involutions and the matrix loci $\ZZZ = \mathcal{PM}_n = \{w \in \MMM_n : w(i) \neq i \text{ for all }i \}$ of fixed-point-free involutions.

In this paper, we consider a matrix locus formed by rook placements. For any positive integer $n$, let $[n]$ be the set of integers $\{1,2,\dots,n\}$. For positive integers $n$ and $m$, a \emph{rook placement} $\RRR$ on the $n\times m$ board is a subset of $[n]\times[m]$ such that $(i,j) \in \RRR$ and $(i',j') \in \RRR$ implies $i \neq i'$ and $j \neq j'$. We call each element $(i,j)\in\RRR$ a \emph{rook} of $\RRR$. Define the \emph{size} $|\RRR|$ of $\RRR$ by
\[|\RRR|\coloneqq \text{the number of rooks of $\RRR$}.\]
Identify each rook placement $\RRR$ on the $n\times m$ board with the $n\times m$ matrix $(a_{i,j})_{1\le i\le n, \,1\le j\le m}$ given by
\[a_{i,j}=\begin{cases}
    1, &\text{if $(i,j)\in\RRR$} \\
    0, &\text{otherwise.}
\end{cases}\]
Write $\ZZZ_{n,m,r}\subseteq\Mat_{n\times m}(\CC)$ for the set of all rook placements on the $n\times m$ board of size $r$, where $r\le\min\{m,n\}$. We are interested in the finite locus of all rook placements with at least $r$ rooks:
\[\UZ_{n,m,r}\coloneqq\bigsqcup_{r^\prime=r}^{\min\{m,n\}}\ZZZ_{n,m,r^\prime}\subseteq\Mat_{n\times m}(\CC).\]

The locus $\UZ_{n,m,r}$ is closed under the action of $\symm_n \times \symm_m$ by row and column permutation, and the quotient $R(\UZ_{n,m,r})$ is a graded $\symm_n \times \symm_m$-module. Our results on this quotient include:
\begin{itemize}
    \item We give an explicit generating set of $\gr \, \II(\UZ_{n,m,r})$. See Definition~\ref{def:ideal} and Theorem~\ref{thm:basis}.
    \item We show (Theorem~\ref{thm:basis}) that $R(\UZ_{n,m,r})$ admits a monomial basis $\es(\RRR)$ indexed by rook placements in $\UZ_{n,m,r}$ (see Definition~\ref{def:shadow-mono}), which is defined by extending each rook placement to a permutation in $\symm_{n+m-r}$ and applying Viennot's shadow construction. We further show that this is the standard monomial basis with respect to diagonal monomial orders (Definition~\ref{def:diag-order}).
    \item We describe the graded $\symm_n \times \symm_m$-module structure of $R(\UZ_{n,m,r})$.
\end{itemize}

The rest of the paper is structured as follows. In Section~\ref{sec:background} we give background material on Gröbner theory, orbit harmonics, and representations of $\symm_n$. In Section~\ref{sec:ideal} we define the ideal $I_{n,m,r}$, and study its algebraic structure. In Section~\ref{sec:basis} we define Viennot's shadow line construction on $\symm_n$ and our extension of this construction to $\UZ_{n,m,r}$. We show that this construction gives the standard monomial basis of $R(\UZ_{n,m,r})$ and prove that the ideal we define in Section~\ref{sec:ideal} is exactly $\gr \, \II(\UZ_{n,m,r})$. In Section~\ref{sec:module} we discuss the structure of the graded $\symm_n \times \symm_m$-module $R(\UZ_{n,m,r})$. We close in Section~\ref{sec:conclusion} with conjectures and possible directions of future research.

\section{Background}\label{sec:background}
\subsection{Gröbner Theory}\label{subsec:grobner}
Let $\xxx_N = (x_1, x_2, \dots x_N)$ be a sequence of $N$ variables, and let $\CC[\xxx_N]$ be the polynomial ring over these variables. A total order $<$ on the set of monomials in $\CC[\xxx_N]$ is called a {\em monomial order} if
\begin{itemize}
    \item $1 \leq m$ for every monomial $m \in \CC[\xxx_N]$
    \item for monomials $m, \, m_1, \, m_2$, we have that $m_1 < m_2$ implies $mm_1 < mm_2$.
\end{itemize}

Let $<$ be a monomial order. For any nonzero polynomial $f \in \CC[\xxx_N]$, the {\em initial monomial} $\text{in}_{<} f$ of $f$ is the largest monomial in $f$ with respect to $<$. Given an ideal $I \subseteq \CC[\xxx_N]$, its {\em initial ideal} is defined to be
\begin{equation}
    \text{in}_{<} I = \langle \text{in}_{<}f : f \in I, f\neq 0 \rangle.
\end{equation}
A monomial $m \in \CC[\xxx_N]$ is called a {\em standard monomial} of $I$ if it is not an element of $\text{in}_{<} I$. It is well known that
\begin{equation}
    \{ m + I : m \text{ is a standard monomial of } I \}
\end{equation}
forms a basis for the vector space $\CC[\xxx_N]/I$. This is called the {\em standard monomial basis}. See \cite{Cox-Little-O'shea} for further details on Gröbner theory.

\subsection{Orbit Harmonics}
Let $\ZZZ \subset \CC^N$ be a finite locus of points in affine $N$-space. The {\em vanishing ideal} $\II(\ZZZ) \subseteq \CC[\xxx_N]$ is the ideal
\begin{equation}
    \II(\ZZZ) = \{ f \in \CC[\xxx_N]: f(z) = 0 \text{ for all } z \in \ZZZ \}. 
\end{equation}

Identifying $\CC[\ZZZ]$ with the vector space of all functions from $\ZZZ$ to $\CC$, we obtain
\begin{equation}\label{eq: locus-vanishing isomorphism}
    \CC[\ZZZ] \cong \CC[\xxx_N]/\II(\ZZZ).
\end{equation}

Given a nonzero polynomial $f \in \CC[\xxx_N]$, write $f = f_0 + f_1 + \dots + f_d$, where $f_i$ is homogeneous of degree $i$ and $f_d \neq 0$. Denote by $\tau(f)$ the top degree homogeneous part of $f$, that is, $\tau(f) = f_d$. For any ideal $I \subseteq \CC[\xxx_N]$, its {\em associated graded ideal} is
\begin{equation}
    \gr \, I = \langle \tau(f):f \in I \rangle.
\end{equation}
\begin{remark}
    Given a set of generators $I = \langle f_1, f_2, \dots, f_r \rangle$, we have the containment of ideals $\langle \tau(f_1),\tau(f_2), \dots, \tau(f_r) \rangle \subseteq \gr \, I$. This containment is strict in general. On the other hand, if $<$ is a monomial order such that $f \leq g$ if and only if $\deg(f) \leq deg(g)$ and $\{ f_1, f_2, \dots, f_r \}$ forms a Gröbner basis of $I$ with respect to $<$, we have the equality $\langle \tau(f_1), \tau(f_2), \dots, \tau(f_r) \rangle = \gr \, I$.
\end{remark}
The isomorphism of ungraded vector spaces in Equation~\eqref{eq: locus-vanishing isomorphism} can be extended to
\begin{equation}\label{eq: orbit-harmmonics}
    \CC[\ZZZ] \cong \CC[\xxx_N]/\II(\ZZZ) \cong \CC[\xxx_N]/\gr \, \II(\ZZZ) =: R(\ZZZ).
\end{equation}
Since the generators of $\gr \, \II(\ZZZ)$ are all homogeneous, $R(\ZZZ)$ has the additional structure of a graded vector space. Furthermore, if $\ZZZ$ is stable under the action of a matrix group $G \subseteq GL_N(\CC)$, Equation ~\eqref{eq: orbit-harmmonics} is also an isomorphism of ungraded $G$-modules, and $R(\ZZZ)$ has the additional structure of a graded $G$-module.

\subsection{Schensted Correspondence}
Let $n$ be any positive integer. A {\em partition} $\lambda$ of $n$ is a sequence of non-increasing positive integers $\lambda = (\lambda_1, \lambda_2, \dots, \lambda_k)$ such that $\sum_{i=1}^k \lambda_i = n$. We write $\lambda \vdash n$ to indicate that $\lambda$ is a partition of $n$.

For a partition $\lambda \vdash n$, its {\em Young diagram} is a diagram with $n$ left-justified boxes, with $\lambda_i$ boxes on the $i$-th row. A {\em standard Young tableau} of shape $\lambda$ is a filling of $[n]$ into the boxes in the Young diagram of $\lambda$, so that the numbers increase across each row and down each column. Below on the left is the Young diagram of the partition $\lambda = (4,2,1) \vdash 7$ and on the right is an example of a standard Young tableau of shape $\lambda$.
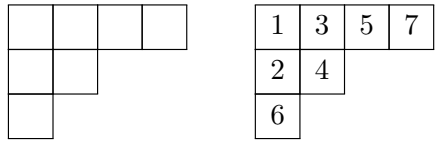
\begin{center}   
\begin{ytableau}
    \: & \: & \: & \: \cr
       & \cr
    \:
\end{ytableau} \quad \quad
\begin{ytableau}
    1 & 3 & 5 & 7 \cr
    2 & 4 \cr
    6
\end{ytableau}

\end{center}
Denote by $\SYT(\lambda)$ the set of standard Young tableaux of shape $\lambda$. The {\em Schensted correspondence} \cite{Schensted_1961} is a bijection
\begin{equation}
    \symm_n \xrightarrow{\quad \sim \quad} \bigsqcup_{\lambda \vdash n} \{ (P,Q): P,Q \in \SYT (\lambda)\} 
\end{equation}
that takes a permutation in $\symm_n$ to a pair of standard Young tableaux of the same shape. This bijection is usually achieved through an insertion algorithm, details of which can be found in ~\cite{sagan2013symmetric}. In Section~\ref{subsec:Viennot-shadow}, we will introduce an alternate description of the Schensted correspondence discovered by Viennot~\cite{Viennot}.

\subsection{Representation Theory}


Let $\Lambda = \bigoplus_{n \, \geq \, 0} \Lambda_n$ be the graded algebra of symmetric functions in an infinite set of variables $\xxx = (x_1, x_2, \dots )$ over the ground field $\CC(q)$. For example, the {\em complete homogeneous symmetric function} of degree $n$ is $h_n := \sum_{i_1 \leq \dots \leq i_n} x_{i_1} \dots x_{i_n} \in \Lambda$ and the {\em elementary symmetric function} of degree $n$ is $e_n = \sum_{i_1 < \dots <i_n} x_{i_1} \dots x_{i_n}$. Both  $\{h_1, h_2, \dots \}$ and $\{e_1, e_2, \dots \}$ are algebraically independent generating sets of $\Lambda$.

The degree-$n$ component $\Lambda_n$ of $\Lambda$ has bases indexed by partitions $\lambda \vdash n$. For example, the {\em complete homogeneous basis} of $\Lambda_n$ is $\{h_{\lambda} \, : \, \lambda \vdash n\}$ where $h_{\lambda} = h_{\lambda_1}h_{\lambda_2}\dots$ for $\lambda = (\lambda_1 \geq \lambda_2 \geq \dots)$. We will also need the {\em Schur basis} $\{ s_\lambda \,:\, \lambda \vdash n\}$ (see \cite{Macdonald} for its definition). If $\lambda = (n)$ one has $s_{(n)} = h_n$.

The algebra of {\em doubly symmetric functions} is the tensor product $\Lambda \otimes_{\CC(q)} \Lambda$, which is naturally doubly graded. The Schur basis of the degree $(n,m)$ piece is $\{s_\lambda \otimes s_\mu \, : \, \lambda \vdash n \, , \, \mu \vdash m \}$.

A symmetric function $F\in\Lambda$ is \emph{Schur-positive} if it can be written as
$F = \sum_{\lambda}c_\lambda(q)\cdot s_\lambda$ with coefficients $c_\lambda(q)\in\mathbb{R}_{\ge 0}[q]$. Similarly, a doubly symmetric function $F^\prime\in\Lambda\otimes_{\CC(q)}\Lambda$ is \emph{Schur-positive} if it can be written as $F^\prime = \sum_{\lambda,\mu}c_{\lambda,\mu}(q)\cdot s_\lambda\otimes s_\mu$ with coefficients $c_{\lambda,\mu}\in\mathbb{R}_{\ge 0}[q]$.
For $F,G\in\Lambda$ (resp. $F,G\in\Lambda\otimes_{\CC(q)}\Lambda$), we write $F\le G$ if and only if $G-F$ is Schur-positive.

We introduce the \emph{truncation operator} $\{-\}_P$ where $P$ is a condition that partitions may or may not satisfy. Given $F=\sum_{\lambda}c_{\lambda} s_\lambda \in \Lambda$ with $c_\lambda\in\CC(q)$, we define $\{F\}_P$ by
\[\{F\}_P\coloneqq\sum_{\text{$\lambda$ satisfies $P$}}c_\lambda s_\lambda.\]
Given $G=\sum_{\lambda,\mu}c_{\lambda,\mu}s_\lambda\otimes s_\mu\in\Lambda\otimes_{\CC(q)}\Lambda$ with $c_{\lambda,\mu}\in\CC(q)$ for all partitions $\lambda,\mu$, we similarly define $\{G\}_P$ by
\[\{G\}_P\coloneqq\sum_{\text{$\lambda,\mu$ both satisfy $P$}}c_{\lambda,\mu}s_\lambda\otimes s_\mu.\]
For instance, we have \[\{G\}_{\lambda_1\le L}=\sum_{\substack{\lambda_1\le L \\\mu_1\le L}}c_{\lambda,\mu}s_\lambda\otimes s_\mu.\]

Irreducible representations of $\symm_n$ are indexed by partitions $\lambda \vdash n$. We write $V^{\lambda}$ for the irreducible $\symm_n$-module indexed by $\lambda$. Given an $\symm_n$-module $V$, there exists a unique decomposition $V \cong \bigoplus_{\lambda \vdash n} c_\lambda V^{\lambda}$ with $c_\lambda \geq 0$. The {\em Frobenius image} of $V$ is defined to be the symmetric function
\begin{equation}
    \Frob(V) := \sum_{\lambda \vdash n} c_\lambda s_\lambda,
\end{equation}
obtained by replacing irreducible representation $V^\lambda$ with the Schur polynomial $s_\lambda$. If $V = \bigoplus_{d \geq 0} V _d$ is a graded $\symm_n$-module, the {\em graded Frobenius image} of $V$ is
\begin{equation}
    \grFrob(V,q) := \sum_{d \geq 0} \Frob(V_d) \cdot q^d.
\end{equation}

For integers $n, m \geq 0$, irreducible representations of $\symm_n \times \symm_m$ are of the form $V^\lambda \otimes V^\mu$, where $\lambda \vdash n$ and $\mu \vdash m$. Each $\symm_n \times \symm_m$-module $V$ has a unique decomposition $V \cong \bigoplus c_{\lambda,\mu} V^\lambda \otimes V^{\mu}$ with $c_{\lambda,\mu} \geq 0$, and we extend the definition of Frobenius images to $\symm_n \times \symm_m$-modules by letting
\begin{equation}
    \Frob(V) := \sum_{\substack{\lambda \vdash n \\ \mu \vdash m}} c_{\lambda,\mu} s_\lambda \otimes s_\mu,
\end{equation}
and define the graded Frobenius images for graded $\symm_n \times \symm_m$-modules analogously.

We need the following result about symmetrizers acting on irreducible representations of $\symm_n$. Let $j < n$ be a positive integer and embed $\symm_j \subseteq \symm_n$ by letting it act on the first $j$ letters. This induces an embedding of group algebras $\CC[\symm_j] \subseteq \CC[\symm_n]$. Let $\eta_j \in \CC[\symm_j] \subseteq \CC[\symm_n]$ be the element that symmetrizes over $\symm_j$, i.e.
\begin{equation}\label{eq:eta_j}
    \eta_j = \sum_{w \in \symm_j}w.
\end{equation}
If $V$ is an $\symm_n$-module, $\eta_j$ acts as an operator on $V$, and we have the following lemma which characterizes when $\eta_j$ annihilates irreducible $\symm_n$-modules.

\begin{lemma}\label{lem:ann-length}
    For $j \leq n$ and $\lambda \vdash n$, we have that $\eta_j \cdot V^\lambda \neq 0$ if and only if $\lambda_1 \geq j$.
\end{lemma}

\begin{proof}
    For any $\symm_n$-module $V^{\lambda}$, we have
    \begin{equation}
        \eta_j \cdot V^{\lambda} = (V^{\lambda})^{\symm_j} = \{v \in V^{\lambda} \, : \, w \cdot v = v  \text{ for all } w \in \symm_j\}.
    \end{equation}
    Thus, $\eta_j \cdot V^{\lambda} \neq 0$ if and only if $(\Res^{\symm_n}_{\symm_j} \,V^{\lambda})^{\symm_j} \neq 0$. This is true if and only if the trivial representation $V^{(j)}$ appears with positive multiplicity in the decomposition of $\Res^{\symm_n}_{\symm_j} \,V^{\lambda}$. The branching rule of $\symm_n$-modules (see~\cite{Macdonald,sagan2013symmetric}) states that
    \begin{equation}\label{eq:branching-rule}
        \Res^{\symm_n}_{\symm_{n-1}} \, V^{\lambda} = \bigoplus_{\substack{\mu \vdash n-1, \\ \mu_i \leq \lambda_i \text{ for all } i}}  V^{\mu},
    \end{equation}
    iterating which implies that $V^{(j)}$ has positive multiplicity in $\Res^{\symm_n}_{\symm_j} \,V^{\lambda}$ if and only if $\lambda_1 \geq j$, completing the proof.
\end{proof}

\section{The ideal $I_{n,m,r}$ and its structure}\label{sec:ideal}

Throughout the remainder of the paper, let $n,m$ be positive integers and $r$ a nonnegative integer with $r \leq \min\{n,m\}$. As in the introduction, $\ZZZ_{n,m,r}$ is the finite locus of rook placements of size $r$ on the $n \times m$ board. We can view $\ZZZ_{n,m,r} \subseteq \Mat_{n \times m}(\CC)$ as the locus of $n \times m$ complex matrices defined by the following conditions:
\begin{itemize}
    \item every entry in the matrix is $0$ or $1$,
    \item there is at most one nonzero entry in each row or column,
    \item there are $r$ ones in total.
\end{itemize}
The locus $\ZZZ_{n,m,r}$ itself is hard to analyze using orbit harmonics. However, it turns out that the locus of {\em upper rook placements}, defined as
\begin{equation}
    \UZ_{n,m,r}\coloneqq\bigsqcup_{r^\prime=r}^{\min\{m,n\}}\ZZZ_{n,m,r^\prime}\subseteq\Mat_{n\times m}(\CC)
\end{equation}
the union of all rook placements of size at least $r$, behaves nicely with the method of orbit harmonics. To study the orbit harmonics quotient $R(\ZZZ_{n,m,r})$, we define the following ideal, which we later prove to be $\gr \, \II(\UZ_{n,m,r})$.
\begin{definition}\label{def:ideal}
    Let $\xxx_{n\times m}$ be an $n\times m$ matrix of variables $(x_{i,j})_{1\le i\le n, 1\le j\le m}$ and consider the polynomial ring $\CC[\xxx_{n\times m}]$ over these variables. Let $I_{n,m,r}\subseteq\CC[\xxx_{n\times m}]$ be the ideal generated by
    \begin{itemize}
        \item any product $x_{i,j}\cdot x_{i,j^\prime} \, (1\le i\le n ; \, 1\le j,j^\prime\le m)$ of variables in the same row,
        \item any product $x_{i,j}\cdot x_{i^\prime,j}\, (1\le i,i'\le n ; \, 1\le j\le m)$ of variables in the same column,
        \item any product $\prod_{k=1}^{n-r+1}\big(\sum_{j=1}^m x_{i_k,j}\big)$ of $n-r+1$ distinct row sums for $1\le i_1<\dots<i_{n-r+1}\le n$, and
        \item any product $\prod_{k=1}^{m-r+1}\big(\sum_{i=1}^n x_{i,j_k}\big)$ of $m-r+1$ distinct column sums for $1\le j_1<\dots<j_{m-r+1}\le m$.
    \end{itemize}
\end{definition}

To study this ideal, we will also need the following definition, which associates a monomial with any rook placement:
\begin{definition}\label{def:monomial}
    For any rook placement $\RRR$ on the $n\times m$ board, the monomial $\mmm(\RRR)\in\CC[\xxx_{n\times m}]$ is defined to be
    \[\mmm(\RRR)\coloneqq\prod_{(i,j)\in\RRR}x_{i,j}.\]
\end{definition}

In the following subsection, we first handle the special case $r=0$, which prepares the reader for the more technical Section~\ref{subsec:ideal-str}.

\subsection{$r=0$ case}\label{subsec: r=0}
When $r=0$, write $\ZZZ_{n,m} := \UZ_{n,m,0}$ for the locus of all rook placements on the $n\times m$ board. The locus $\ZZZ_{n,m}$ is a finite {\em shifted locus}, which was defined and studied by Reiner and Rhoades~\cite[Sec. 3.4]{reiner2024harmonicsgradedehrharttheory}. Lemma~\ref{lem:toy-span} and Proposition ~\ref{prop:toy-basis} below can be derived directly from the results by Reiner and Rhoades, but we still include their proofs as they will promote understanding of later sections. 

Write $I_{n,m} := I_{n,m,0}$ for the ideal generated by
\begin{itemize}
    \item any product $x_{i,j}\cdot x_{i,j^\prime}$ for $1\le i\le n$ and $1\le j,j^\prime\le m$ of variables in the same row, 
    \item any product $x_{i,j}\cdot x_{i^\prime,j}$ for $1\le i, i^\prime \le n$ and $1\le j\le m$ of variables in the same column.
\end{itemize}
Note that the products of row sums and column sums are no longer present as it's impossible to have $n+1$ distinct rows or $m+1$ distinct columns. We start by giving a spanning set of the vector space $\CC[\xxx_{n\times m}]/I_{n,m}$:
\begin{lemma}\label{lem:toy-span}
    The family of monomials $\{\mmm(\RRR)\,:\,\RRR\in\ZZZ_{n,m}\}$ descends to a spanning set of the vector space $\CC[\xxx_{n\times m}]/I_{n,m}$.
\end{lemma}
\begin{proof}
    Note that the generators $x_{i,j}\cdot x_{i,j^\prime}$ and $x_{i,j}\cdot x_{i^\prime ,j}$ of $I_{n,m}$ above imply that all the monomials containing variables from the same column or row of $\xxx_{n\times m}$ vanish in the quotient ring $\CC[\xxx_{n\times m}]/I_{n,m}$. Consequently, the monomials with no two variables from the same row or column span $\CC[\xxx_{n\times m}]/I_{n,m}$.
\end{proof}

We next strengthen Lemma~\ref{lem:toy-span} to yield a basis of $\CC[\xxx_{n\times m}]/I_{n,m}$.
\begin{proposition}\label{prop:toy-basis}
    The family of monomials $\{\mmm(\RRR)\,:\,\RRR\in\ZZZ_{n,m}\}$ descends to a  $\CC$-linear basis of $\CC[\xxx_{n\times m}]/I_{n,m}$. Furthermore, we have $\gr \, \II(\ZZZ_{n,m}) = I_{n,m}$.
\end{proposition}
\begin{proof}
Since a rook placement contains no rooks in the same row or column, it is clear that $x_{i,j}\cdot x_{i,j^\prime}\in\II(\ZZZ_{n,m})$ and $x_{i,j}\cdot x_{i^\prime,j}\in\II(\ZZZ_{n,m})$. We thus have $x_{i,j}\cdot x_{i,j^\prime}\in\gr \, \II(\ZZZ_{n,m})$ and $x_{i,j}\cdot x_{i^\prime,j}\in\gr \, \II(\ZZZ_{n,m})$. This implies that $I_{n,m}\subseteq\gr \, \II(\ZZZ_{n,m})$, which naturally induces a linear surjection \[\CC[\xxx_{n\times m}]/I_{n,m}\twoheadrightarrow R(\ZZZ_{n,m}).\]

We now compare dimensions. Lemma~\ref{lem:toy-span} implies that \[\dim_\CC (\CC[\xxx_{n\times m}]/I_{n,m})\le|\ZZZ_{n,m}| = \dim_\CC (R(\ZZZ_{n,m})),\] while the surjection above implies $\dim_\CC (\CC[\xxx_{n\times m}]/I_{n,m})\ge \dim_\CC (R(\ZZZ_{n,m}))$. Hence, all inequalities are equalities, and the surjection above is a linear isomorphism. This forces $I_{n,m} = \gr \, \II(\ZZZ_{n,m})$ and ensures that the spanning set $\{\mmm(\RRR)\,:\,\RRR\in\ZZZ_{n,m}\}$ in Lemma~\ref{lem:toy-span} is a basis.
\end{proof}

The monomial basis from Proposition~\ref{prop:toy-basis} inherits the $\symm_n\times\symm_m$-action on $\ZZZ_{n,m}$, allowing us to compute the graded module structure of $R(\ZZZ_{n,m}) = \CC[\xxx_{n\times m}]/I_{n,m}$.
\begin{proposition}\label{prop:toy-module-str}
    The graded Frobenius image of $\CC[\xxx_{n\times m}]/I_{n,m}$ is given by
    \[\grFrob(\CC[\xxx_{n\times m}]/I_{n,m};q) = \sum_{d=0}^{\min\{m,n\}} q^d \cdot \bigg(\sum_{\mu\vdash d}(s_{\mu}\cdot h_{n-d})\otimes(s_\mu\cdot h_{m-d})\bigg).\]
\end{proposition}
\begin{proof}
    Note that the monomial basis in Proposition~\ref{prop:toy-basis} is compatible with the $\symm_n\times\symm_m$-action on $\ZZZ_{n,m}$. Therefore, we have an isomorphism of $\symm_n\times\symm_m$-modules:
    
    \[(\CC[\xxx_{n\times m}]/I_{n,m})_d\cong\CC[\ZZZ_{n,m,d}]\]
    for $0\le d\le\min\{m,n\}$, and
    \[(\CC[\xxx_{n\times m}]/I_{n,m})_d = \{0\}\]
    for $d>\min\{m,n\}$.

    It suffices to study the module structure of $\CC[\ZZZ_{n,m,d}]$ for $0\le d \le\min\{m,n\}$. Consider $\symm_{d}\times\symm_{n-d}$ (resp. $\symm_d\times\symm_{m-d}$) as a subgroup of $\symm_n$ (resp. $\symm_m$) by letting $\symm_d$ and $\symm_{n-d}$ (resp. $\symm_{m-d}$) separately act on $\{1,\dots,d\}$ and $\{d+1,\dots,n\}$ (resp. $\{d+1,\dots,m\}$). Thus, we embed $(\symm_d\times\symm_{n-d})\times(\symm_{d}\times\symm_{m-d})$ into $\symm_n\times\symm_m$ as a subgroup. Note that $(\symm_d\times\symm_{n-d})\times(\symm_{d}\times\symm_{m-d})$ acts on $\CC[\symm_d]$ by
    \[((g_1,h_1),(g_2,h_2))\cdot g = g_1 g g_2^{-1}\]
    which restricts to the $\symm_d\times\symm_d$-action. Artin-Wedderburn Theorem gives this $\symm_d\times\symm_d$-module structure by
    \begin{align}\label{eq:artin-wed}
    \CC[\symm_d]\cong\bigoplus_{\mu\vdash d}V^\mu\otimes V^\mu
    \end{align}
    where $\symm_d\times\symm_d$ acts on each $V^\mu\otimes V^\mu$ by
    $(g_1,g_2)\cdot (v_1\otimes v_2) = (g_1\cdot v_1)\otimes(g_2\cdot v_2)$. Then we have that, as $\symm_n\times\symm_m$-modules, 
    \begin{align*}
        \CC[\ZZZ_{n,m,d}] &\cong \ind_{(\symm_d\times\symm_{n-d})\times(\symm_{d}\times\symm_{m-d})}^{\symm_n\times\symm_m}\CC[\symm_d] \\ 
        &\cong \ind_{(\symm_d\times\symm_{n-d})\times(\symm_{d}\times\symm_{m-d})}^{\symm_n\times\symm_m}\bigg(\bigoplus_{\mu\vdash d}(V^\mu\otimes V^{(n-d)})\otimes( V^\mu \otimes V^{(m-d)})\bigg) \\
        &\cong \bigoplus_{\mu\vdash d}\ind_{(\symm_d\times\symm_{n-d})\times(\symm_{d}\times\symm_{m-d})}^{\symm_n\times\symm_m}\big((V^\mu\otimes V^{(n-d)})\otimes( V^\mu \otimes V^{(m-d)})\big) \\
        &\cong \bigoplus_{\mu\vdash d}\Big(\ind_{\symm_d\times\symm_{n-d}}^{\symm_n}(V^\mu\otimes V^{(n-d)})\Big)\otimes\Big(\ind_{\symm_d\times\symm_{m-d}}^{\symm_m}(V^\mu\otimes V^{(m-d)})\Big)
    \end{align*}
    where the second isomorphism arises from Equation~\eqref{eq:artin-wed}. Therefore, the Frobenius image of $\CC[\ZZZ_{n,m,d}]$ as an $\symm_n\times\symm_m$-module is given by
    \begin{equation}\label{eq:graded-component-module-str=rook}
    \Frob(\CC[\ZZZ_{n,m,d}]) = \sum_{\mu\vdash d}(s_\mu\cdot h_{n-d})\otimes(s_{\mu}\cdot h_{m-d})
    \end{equation}
    and hence
    \[\Frob\big((\CC[\xxx_{n\times m}])_d\big) = \sum_{\mu\vdash d}(s_\mu\cdot h_{n-d})\otimes(s_{\mu}\cdot h_{m-d}),\]
    completing our proof.
\end{proof}

\subsection{Structure of $I_{n,m,r}$}\label{subsec:ideal-str}

In order to analyze the standard monomial basis of $\CC[\xxx_{n\times m}]/I_{n,m,r}$ in Section~\ref{sec:basis} and the graded module structure of $\CC[\xxx_{n\times m}]/I_{n,m,r}$ in Section~\ref{sec:module}, we start by studying strategic elements of the ideal $I_{n,m,r}$. On a first reading, it may be better to skip all the technical proofs in this section and accept  Corollary~\ref{cor:ann-of-I_{n,m,r}-extend} on faith.

For induction convenience, we introduce a new ideal $J_{n,m,r}\subseteq I_{n,m,r}$ by removing some generators of $I_{n,m,r}$ in Definition~\ref{def:ideal}:
\begin{definition}\label{def:intermidiate-ideal}
    Let $J_{n,m,r}\subseteq\CC[\xxx_{n\times m}]$ be the ideal generated by
    \begin{itemize}
        \item all products $x_{i,j}\cdot x_{i,j^{\prime}}$ for $1\le i\le n$ and $1\le j,j^\prime\le m$ of variables in the same row, and
        \item all products $\prod_{k=1}^{m-r+1}\big(\sum_{i=1}^n x_{i,j_k}\big)$ of $m-r+1$ distinct column sums for $1\le j_1<\dots<j_{m-r+1}\le m$.
    \end{itemize}
\end{definition}

We start by modifying the products of column sums in Definition~\ref{def:intermidiate-ideal} so they are easier to use:
\begin{lemma}\label{lem:generator-slight-modify}
    If $m-r<t\le m$ and $b_1,\dots,b_t\in[m]$ are distinct integers, then
    \[\sum_{\substack{i_1,\dots,i_t\in[n]\\\text{distinct}}}\prod_{j=1}^t x_{i_j,b_j} \in J_{n,m,r}\]
    summing over all ordered sequences of distinct integers in $[n]$ of length $t$. 
\end{lemma}
\begin{proof}
    Since any product of the form $x_{i,j}\cdot x_{i,j^\prime}$ belongs to $J_{n,m,r}$, we have that
    \[\sum_{\substack{i_1,\dots,i_t\in[n]\\\text{distinct}}}\prod_{j=1}^t x_{i_j,b_j} \equiv \sum_{i_1,\dots,i_t \in [n]}\prod_{j=1}^t x_{i_j,b_j} =\prod_{j=1}^t \sum_{i=1}^n x_{i,b_j} \mod{J_{n,m,r}}.\]
    Note that $t>m-r$. Then the generators of $J_{n,m,r}$ of the form $\prod_{k=1}^{m-r+1}\big(\sum_{i=1}^n x_{i,j_k}\big)$ imply that $\prod_{j=1}^t \sum_{i=1}^n x_{i,b_j} \in J_{n,m,r}$, completing the proof.
\end{proof}

We will then prove the most technical result in this section. As we have the $\CC$-algebra identification $\CC[\symm_n\times\symm_m]\cong\CC[\symm_n]\otimes_\CC\CC[\symm_m]$, we write $\eta_p\otimes 1$ for $\sum_{w\in\symm_p}(w,1)$.

\begin{lemma}\label{lem:ideal-induction}
    Let $p$ and $d$ be two integers such that $0\le d\le\min\{m,n\}$ and $n+m-d-r<p\le n$. Then, for any rook placement $\RRR\in\ZZZ_{n,m,d}$, we have
    \[(\eta_p \otimes 1)\cdot\mmm(\RRR)\in J_{n,m,r}.\]
\end{lemma}
\begin{proof}
    We prove the lemma by inducting on $n$. The base case $n=1$ is trivial. Assume that Lemma~\ref{lem:ideal-induction} holds for any positive integer $n^\prime<n$, we prove it for $n$.

    Write $\mmm(\RRR) = \prod_{k=1}^d x_{a_k,b_k}$ where $b_1,\dots,b_d\in[m]$ are distinct and $1\le a_1<\dots< a_d\le n$. Since $p+d>n+m-r\ge n$, it follows that $[p]\cap\{a_1,\dots,a_d\}\neq\varnothing$. Therefore, we can choose the largest $t\in[d]$ such that $a_t\in[p]$.

    \textbf{Claim:} $t>m-r$.

    In fact, the maximality of $t$ implies that $[p]\cap\{a_{t+1},\dots,a_d\} = \varnothing$, so we have that $[p]\sqcup\{a_{t+1},\dots,a_d\}\subseteq[n]$ and hence $p+d-t\le n$. Consequently, $t \ge p+d-n > m-r$, which completes the proof of the claim.

    Back to the proof of Lemma~\ref{lem:ideal-induction}. We have that
    \begin{align*}
        &(\eta_p\otimes 1)\cdot\mmm(\RRR) = (\eta_p\otimes 1)\cdot\prod_{k=1}^d x_{a_k,b_k} = (p-t)!\cdot\Bigg(\sum_{\substack{i_1,\dots,i_t\in[p]\\\text{distinct}}}x_{i_1,b_1}\dots x_{i_t,b_t}\Bigg)\cdot\prod_{k=t+1}^d x_{a_k,b_k} \\
        \equiv & (p-t)!\cdot \Bigg(-\sum_{\substack{\text{distinct $i_1,\dots,i_t\in[n]$,}\\ \{i_1,\dots,i_t\}\not\subseteq[p]}}x_{i_1,b_1}\dots x_{i_t,b_t}\Bigg)\cdot\prod_{k=t+1}^d x_{a_k,b_k} \\
        \equiv & (p-t)!\cdot \Bigg(-\sum_{\substack{\text{distinct $i_1,\dots,i_t\in[n]\setminus\{a_{t+1},\dots,a_d\}$,}\\ \{i_1,\dots,i_t\}\not\subseteq[p]}}x_{i_1,b_1}\dots x_{i_t,b_t}\Bigg)\cdot\prod_{k=t+1}^d x_{a_k,b_k} \mod{J_{n,m,r}}
    \end{align*}
    where the first congruence follows from the claim above, and the second from the relation $x_{i,j}\cdot x_{i,j^\prime}\in J_{n,m,r}$. Now it remains to show that
    \begin{align}\label{eq:ideal-induction}
        \Bigg(\sum_{\substack{\text{distinct $i_1,\dots,i_t\in[n]\setminus\{a_{t+1},\dots,a_d\}$,}\\ \{i_1,\dots,i_t\}\not\subseteq[p]}}x_{i_1,b_1}\dots x_{i_t,b_t}\Bigg)\cdot\prod_{k=t+1}^d x_{a_k,b_k} \in J_{n,m,r}.
    \end{align}
    If the summation in \eqref{eq:ideal-induction} is empty, there is nothing to prove. Otherwise, we decompose the summation into sub-summations with smaller index sets as follows. Note that the summation in \eqref{eq:ideal-induction} requires that $\{i_1,\dots,i_t\}\not\subseteq[p]$, so we assume, without loss of generality, that $\{i_1,\dots,i_t\}\setminus[p]=\{i_{t-s+1},\dots,i_t\}$ where $0<s<t$. (Here, we have $s<t$ because $\{i_1,\dots,i_t\}\cap[p]\neq\varnothing$. Otherwise, $\{i_1,\dots,i_t\}\sqcup[p]\sqcup\{a_{t+1},\dots,a_d\}\subseteq[n]$ implies that $t+p+d-t\le n$, i.e. $p\le n-d$. However, $p>n+m-d-r\ge n-d$, a contradiction.) Summing over such $i_1,\dots,i_t$ for fixed $i_{t-s+1}=c_{t-s+1},\dots,i_t=c_t$, we obtain a sub-summation of the summation in \eqref{eq:ideal-induction}. It suffices to show \eqref{eq:ideal-induction} after replacing its summation with these sub-summations, since adding them up yields \eqref{eq:ideal-induction} itself. Consequently, it remains to show that
    \[\Bigg(\sum_{\substack{i_1,\dots,i_{t-s}\in[p]\setminus\{a_{t+1},\dots,a_d\}\\ \text{distinct}}}x_{i_1,b_1}\dots x_{i_{t-s},b_{t-s}}\Bigg)\cdot\prod_{k=t-s+1}^t x_{c_k,b_k}\cdot\prod_{k=t+1}^d x_{a_k,b_k} \in J_{n,m,r}.\] Recall that the choice of $t$ implies that $[p]\cap\{a_{t+1},\dots,a_d\}=\varnothing$, so it remains to show that
    \begin{equation}\label{eq:ideal-induction-final}
        \Bigg(\sum_{\substack{i_1,\dots,i_{t-s}\in[p]\\ \text{distinct}}}x_{i_1,b_1}\dots x_{i_{t-s},b_{t-s}}\Bigg)\cdot\prod_{k=t-s+1}^t x_{c_k,b_k}\cdot\prod_{k=t+1}^d x_{a_k,b_k} \in J_{n,m,r}.
    \end{equation}
    Note that this expression possesses the following form:
    \begin{equation}\label{eq:rewrite}c\cdot((\eta_p\otimes 1)\cdot\mmm(\RRR^\prime))\cdot\prod_{k=t-s+1}^t x_{c_k,b_k}\end{equation} where $c\in\CC$, and $\RRR^\prime$ is a rook placement of size $d^\prime = d-s$ on the sub-board $([n]\setminus\{c_{t-s+1},\dots,c_t\})\times([m]\setminus\{b_{t-s+1},\dots,b_t\})$ of the original board $[n]\times[m]$. For the sake of induction, consider restricting Definition~\ref{def:intermidiate-ideal} to a smaller variable set:
    \begin{itemize}
        \item $n^\prime=n-s$, $m^\prime=m-s$, and $r^{\prime}=r-s$;
        \item the restricted variable matrix given by \[\xxx_{n^\prime,m^\prime}\coloneqq(x_{i,j})_{i\in[n]\setminus\{c_{t-s+1},\dots,c_t\},j\in[m]\setminus\{b_{t-s+1},\dots,b_t\}};\]
        \item the restricted counterpart $J^\prime_{n^\prime,m^\prime,r^\prime}\subseteq\CC[\xxx]$ of $J_{n,m,r}$ generated by
        \begin{itemize}
            \item any product $x_{i,j}\cdot x_{i,j^{\prime}}$ for $i\in[n]\setminus\{c_{t-s+1},\dots,c_t\}$ and $j,j^\prime\in[m]\setminus\{b_{t-s+1},\dots,x_t\}$, and
            \item any product $\prod_{k=1}^{m^\prime-r^\prime+1}(\sum_{i\in[n]\setminus\{c_{t-s+1},\dots,c_t\}}x_{i,j_k})$ for distinct $j_1,\dots,j_{m^\prime-r^\prime+1}\in[m]\setminus\{b_{t-s+1},\dots,b_t\}$.
        \end{itemize}
    \end{itemize}
    Note that $p>n+m-d-r=(n-s)+(m-s)-(d-s)-(r-s)=n^\prime+m^\prime-d^\prime-r^\prime$, and hence the induction assumption reveals that $(\eta_p\otimes 1)\cdot\mmm(\RRR^\prime)\in J_{n^\prime,m^\prime,r^\prime}^\prime$. Therefore, in order to show \eqref{eq:ideal-induction-final} using Expression~\eqref{eq:rewrite}, it suffices to show that
    \[J_{n^\prime,m^\prime,r^\prime}^\prime\cdot\prod_{k=t-s+1}^t x_{c_k,b_k} \subseteq J_{n,m,r}.\] 
    We verify this containment using the generators of $J_{n^\prime,m^\prime,r^\prime}^\prime$ above. For any product $x_{i,j}\cdot x_{i,j^{\prime}}$ for $i\in[n]\setminus\{c_{t-s+1},\dots,c_t\}$ and $j,j^\prime\in[m]\setminus\{b_{t-s+1},\dots,x_t\}$, it is clear that $x_{i,j}\cdot x_{i,j^{\prime}}\cdot\prod_{k=t-s+1}^t x_{c_k,b_k}\in J_{n,m,r}$. For any product $\prod_{k=1}^{m^\prime-r^\prime+1}(\sum_{i\in[n]\setminus\{c_{t-s+1},\dots,c_t\}}x_{i,j_k})$ for distinct $j_1,\dots,j_{m^\prime-r^\prime+1}\in[m]\setminus\{b_{t-s+1},\dots,b_t\}$, $m^\prime-r^\prime+1=(m-s)-(r-s)+1=m-r+1$ implies that
    \begin{align*}
        &\prod_{k=1}^{m^\prime-r^\prime+1}\Bigg(\sum_{i\in[n]\setminus\{c_{t-s+1},\dots,c_t\}}x_{i,j_k}\Bigg)\cdot\prod_{k=t-s+1}^t x_{c_k,b_k} \\ = &\prod_{k=1}^{m-r+1}\Bigg(\sum_{i\in[n]\setminus\{c_{t-s+1},\dots,c_t\}}x_{i,j_k}\Bigg)\cdot\prod_{k=t-s+1}^t x_{c_k,b_k} \\
        \equiv & \prod_{k=1}^{m-r+1}\Bigg(\sum_{i\in[n]}x_{i,j_k}\Bigg)\cdot\prod_{k=t-s+1}^t x_{c_k,b_k} \equiv 0 \mod{J_{n,m,r}}
    \end{align*}
    which completes the proof.
\end{proof}

Let $0\le d\le\min\{m,n\}$ be an integer. For any set $S\subseteq[n]$ and $T\subseteq[m]$ such that $|S|,|T|>n+m-d-r$, let $\symm_S\subseteq\symm_n$ (resp. $\symm_T\subseteq\symm_m$) be the symmetric group of bijections $S \to S$ (resp. $T \to T$). We have that the symmetrizers of $\symm_S$ and $\symm_T$ annihilate $\CC[\xxx_{n \times m}] / I_{n,m,r}$.
\begin{corollary}\label{cor:ann-of-I_{n,m,r}-extend}
     For any rook placement $\RRR\in\ZZZ_{n,m,d}$, we have
    \[\sum_{w\in\symm_S}(w,1)\cdot\mmm(\RRR) \in I_{n,m,r}\]
    and
    \[\sum_{w\in\symm_T}(1,w)\cdot\mmm(\RRR) \in I_{n,m,r}.\]
\end{corollary}
\begin{proof}
    write $p = |S|$ and $q=|t|$. Since $J_{n,m,r} \subseteq i_{n,m,r}$, Lemma~\ref{lem:ideal-induction} implies that $(\eta_p \otimes 1)\cdot\mmm(\RRR) \in I_{n,m,r}$. Analogously, we have $(1\otimes\eta_q)\cdot\mmm(\RRR) \in I_{n,m,r}$. 
    The proof is thus complete as $I_{n,m,r}$ is closed under row and column permutations.
\end{proof}

\section{Hilbert series and standard monomial basis}\label{sec:basis}
\subsection{Viennot Shadow}\label{subsec:Viennot-shadow}
We begin this section by describing Viennot's shadow line construction~\cite{Viennot}. Let $w = [w(1),w(2), \dots, w(n)] \in \symm_n$ be a permutation. We draw its diagram on an $n \times n$ grid by putting a point at $(i, w(i))$ for $i \in [n]$. For example, the permutation $w = [6,3,5,7,1,2,8,4] \in \symm_8$ has the digram shown below.
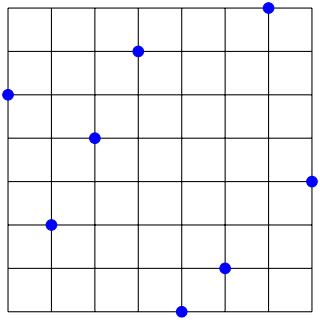
\begin{center}
    \begin{tikzpicture}[x=1.5em,y=1.5em]
    \draw[step=1,black,thin] (0,0) grid (7,7);
    \filldraw [blue] (0,5) circle (2pt);
    \filldraw [blue] (1,2) circle (2pt);
    \filldraw [blue] (2,4) circle (2pt);
    \filldraw [blue] (3,6) circle (2pt);
    \filldraw [blue] (4,0) circle (2pt);
    \filldraw [blue] (5,1) circle (2pt);
    \filldraw [blue] (6,7) circle (2pt);
    \filldraw [blue] (7,3) circle (2pt);
    \end{tikzpicture} 
\end{center}

Place an imaginary light source at the bottom left corner of the diagram, shining northeast. Each point $(i,w(i))$ in the diagram of $w$ blocks the regions to its north and to its east. We take the boundary of the shaded region and call it {\em the first shadow line} $L_1$ (see the left diagram below). Removing all points on the first shadow line and iterating this process yields the {\em second shadow line} $L_2$, {\em third shadow line} $L_3$, and so on (see the center diagram below).
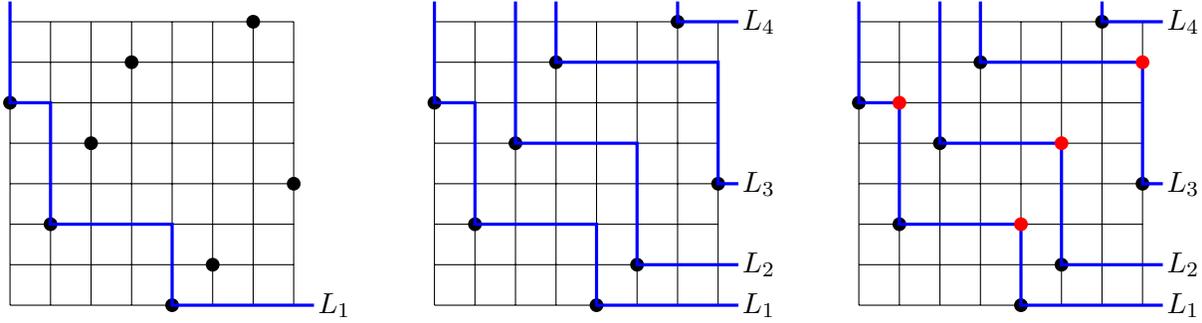
\begin{center}
    \begin{tikzpicture}[x=1.4em,y=1.4em, very thick,color = blue]
    \draw[step=1,black,thin] (0,0) grid (7,7);
    \filldraw [black] (0,5) circle (2pt);
    \filldraw [black] (1,2) circle (2pt);
    \filldraw [black] (2,4) circle (2pt);
    \filldraw [black] (3,6) circle (2pt);
    \filldraw [black] (4,0) circle (2pt);
    \filldraw [black] (5,1) circle (2pt);
    \filldraw [black] (6,7) circle (2pt);
    \filldraw [black] (7,3) circle (2pt);
    \draw(0,7.5)--(0,5)--(1,5)--(1,2)--(4,2)--(4,0)--(7.5,0);
    \node [black] at (8,0) {$L_1$};
    \end{tikzpicture} \quad \quad
    \begin{tikzpicture}[x=1.4em,y=1.4em, very thick,color = blue]
    \draw[step=1,black,thin] (0,0) grid (7,7);
    \filldraw [black] (0,5) circle (2pt);
    \filldraw [black] (1,2) circle (2pt);
    \filldraw [black] (2,4) circle (2pt);
    \filldraw [black] (3,6) circle (2pt);
    \filldraw [black] (4,0) circle (2pt);
    \filldraw [black] (5,1) circle (2pt);
    \filldraw [black] (6,7) circle (2pt);
    \filldraw [black] (7,3) circle (2pt);
    \draw(0,7.5)--(0,5)--(1,5)--(1,2)--(4,2)--(4,0)--(7.5,0);
    \draw(2,7.5)--(2,4)--(5,4)--(5,1)--(7.5,1);
    \draw(3,7.5)--(3,6)--(7,6)--(7,3)--(7.5,3);
    \draw(6,7.5)--(6,7)--(7.5,7);
    \node [black] at (8,0) {$L_1$};
    \node [black] at (8,1) {$L_2$};
    \node [black] at (8,3) {$L_3$};
    \node [black] at (8,7) {$L_4$};
    \end{tikzpicture}
    \quad \quad
    \begin{tikzpicture}[x=1.4em,y=1.4em, very thick,color = blue]
    \draw[step=1,black,thin] (0,0) grid (7,7);
    \filldraw [black] (0,5) circle (2pt);
    \filldraw [black] (1,2) circle (2pt);
    \filldraw [black] (2,4) circle (2pt);
    \filldraw [black] (3,6) circle (2pt);
    \filldraw [black] (4,0) circle (2pt);
    \filldraw [black] (5,1) circle (2pt);
    \filldraw [black] (6,7) circle (2pt);
    \filldraw [black] (7,3) circle (2pt);
    \draw(0,7.5)--(0,5)--(1,5)--(1,2)--(4,2)--(4,0)--(7.5,0);
    \draw(2,7.5)--(2,4)--(5,4)--(5,1)--(7.5,1);
    \draw(3,7.5)--(3,6)--(7,6)--(7,3)--(7.5,3);
    \draw(6,7.5)--(6,7)--(7.5,7);
    \node [black] at (8,0) {$L_1$};
    \node [black] at (8,1) {$L_2$};
    \node [black] at (8,3) {$L_3$};
    \node [black] at (8,7) {$L_4$};
    \filldraw[red] (1,5) circle (2pt)
    (4,2) circle (2pt)
    (5,4) circle (2pt)
    (7,6) circle (2pt);
    \end{tikzpicture}
\end{center}

Suppose that $w \rightarrow (P(w),Q(w))$ under the Schensted correspondence, and that the shadow lines of $w$ are $L_1,L_2,\dots ,L_k$. Viennot proved ~\cite{Viennot} that the entries in the first row of $P(w)$ are the coordinates of the infinite horizontal rays of the shadow lines, and the entries in the first row of $Q(w)$ are the coordinates of the infinite vertical rays. In the example above, we have that the first row of $P(w)$ is $\begin{array}{|c|c|c|c|} 
\hline
 1& 2 & 4 & 8\\
\hline
\end{array}$ and the first row of $Q(w)$ is $\begin{array}{|c|c|c|c|} 
\hline
 1& 3 & 4 & 7\\
\hline
\end{array}$. 
\begin{definition}
     Given a permutation $w \in \symm_n$, the {\em shadow set $\SSS(w)$} of $w$ is the collection of points that lie on the northeast corners of the shadow lines.
\end{definition}
In the above example, $\SSS(w) = \{(2,6),(5,3),(6,5),(8,7)\}$, as depicted in red in the right diagram above. We can iterate the shadow line construction on $\SSS(w)$, as depicted in the diagram below. Viennot proved~\cite{Viennot} that the coordinates of the infinite horizontal rays of the new shadow lines correspond to the second row of $P(w)$, and the $x$ coordinates of the infinite vertical rays correspond to the second row of $Q(w)$. In our example, the second row of $P(w)$ is $\begin{array}{|c|c|c|} 
\hline
 3& 5 & 7 \\
\hline
\end{array}$ and the second row of $Q(w)$ is $\begin{array}{|c|c|c|}
\hline
 2& 6 & 8 \\
\hline
\end{array}$.
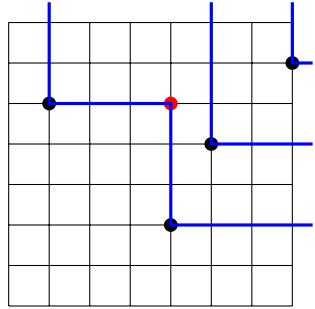
\begin{center}
    \begin{tikzpicture}[x=1.4em,y=1.4em, very thick,color = blue]
    \draw[step=1,black,thin] (0,0) grid (7,7);
    \filldraw [black] (1,5) circle (2pt);
    \filldraw [black] (4,2) circle (2pt);
    \filldraw [black] (5,4) circle (2pt);
    \filldraw [black] (7,6) circle (2pt);
    \filldraw [red] (4,5) circle (2pt);
    \draw(1,7.5)--(1,5)--(4,5)--(4,2)--(7.5,2);
    \draw(5,7.5)--(5,4)--(7.5,4);
    \draw(7,7.5)--(7,6)--(7.5,6);
    \end{tikzpicture} \quad \quad
\end{center}

This iterated process produces the iterated shadow set $\SSS(\SSS(w))$, which consists of a single point as drawn in red above. Repeating this process, we obtain a single shadow line and an empty shadow set,
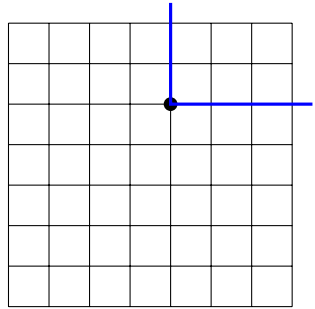
\begin{center}
    
    \begin{tikzpicture}[x=1.4em,y=1.4em, very thick,color = blue]
    \draw[step=1,black,thin] (0,0) grid (7,7);
    \filldraw [black] (4,5) circle (2pt);
    \draw(4,7.5)--(4,5)--(7.5,5);
    \end{tikzpicture}
\end{center}
from which we conclude that $P(w)$ and $Q(w)$ are respectively given by
\\
\begin{center}   
\begin{ytableau}
    1 & 2 & 4 & 8 \cr
    3 & 5 & 7 \cr
    6
\end{ytableau} \quad  \text{and} \quad
\begin{ytableau}
    1 & 3 & 4 & 7 \cr
    2 & 6 & 8\cr
    5
\end{ytableau}

\end{center}
    
From the construction, it is immediate that the shadow set $\SSS(w)$ of any permutation $w$ is a rook placement. However, not all rook placements are shadow sets of permutations. In ~\cite{rhoades2024increasing}, Rhoades described the following algorithm to determine whether a rook placement is the shadow set of some permutation. In the next section, we present a generalized version of this algorithm.
\begin{lemma}\label{lem: shadow-set-dertermine}
    Let $\RRR$ be a rook placement on the $n\times n$ board and apply the Viennot shadow construction to $\RRR$, yielding shadow lines $L_1,\dots,L_p$. Define two sequences $(x_i)_{i=1}^n$ and $(y_j)_{j=1}^m$ over the alphabet $\{-1,0,1\}$ by
    \[x_i=\begin{cases}
        1, &\text{if one of the shadow lines $L_1,\dots,L_p$ has a vertical ray at $x=i$,} \\
        -1, &\text{if the vertical line $x=i$ does not meet $\RRR$,} \\
        0, &\text{otherwise.}
    \end{cases}\]
    and
    \[y_j=\begin{cases}
        1, &\text{if one of the shadow lines $L_1,\dots,L_p$ has a horizontal ray at $y=j$,} \\
        -1, &\text{if the horizontal line $y=j$ does not meet $\RRR$,} \\
        0, &\text{otherwise.}
    \end{cases}\]
    Then $\RRR=\SSS(w)$ is the extended shadow set of some permutation $w \in \symm_n$ if and only if for all $1\le i\le n$ $x_1+x_2+\dots+x_i\le 0$ and $y_1+y_2+\dots+y_j\le 0$.
\end{lemma}

\subsection{Extended shadow sets}\label{subsec:shadow-set}

Recall that $\UZ_{n,m,r}$ is the set of rook placements with at least $r$ rooks. Given a rook placement $\RRR \in \UZ_{n,m,r}$, we define an algorithm that extends $\RRR$ to a permutation $\ex(\RRR) \in \symm_{n+m-r}$. The shadow set of this permutation, denoted by $\eshadow(R)$, plays a key role in analyzing the orbit harmonics quotient $R(\UZ_{n,m,r})$.

To start, let $\RRR \in \UZ_{n,m,r}$ with $r' \geq r$ rooks. Draw the diagram of the rook placement on the $[n] \times [m]$ grid, with a point at $(i,j)$ if $(i,j) \in R$. Below is the diagram of $\RRR = \{(2,3),(3,4),(5,2),(8,5)\} \in \UZ_{8,6,2}$.  

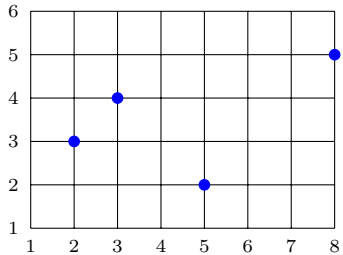
\begin{center}
    \begin{tikzpicture}[x=1.5em,y=1.5em]
    \draw[step=1,black,thin] (0,0) grid (7,5);
    \filldraw [blue] (1,2) circle (2pt);
    \filldraw [blue] (2,3) circle (2pt);
    \filldraw [blue] (4,1) circle (2pt);
    \filldraw [blue] (7,4) circle (2pt);
    \node [black] at (0,-.4) {\tiny $1$};
    \node [black] at (1,-.4) {\tiny $2$};
    \node [black] at (2,-.4) {\tiny $3$};
    \node [black] at (3,-.4) {\tiny $4$};
    \node [black] at (4,-.4) {\tiny $5$};
    \node [black] at (5,-.4) {\tiny $6$};
    \node [black] at (6,-.4) {\tiny $7$};
    \node [black] at (7,-.4) {\tiny $8$};
    \node [black] at (-.4,0) {\tiny $1$};
    \node [black] at (-.4,1) {\tiny $2$};
    \node [black] at (-.4,2) {\tiny $3$};
    \node [black] at (-.4,3) {\tiny $4$};
    \node [black] at (-.4,4) {\tiny $5$};
    \node [black] at (-.4,5) {\tiny $6$};
    \end{tikzpicture}
\end{center}

We extend the diagram to an $(n+m-r) \times (n+m-r)$ grid by adding $(m-r)$ columns on the left and $(n-r)$ rows on the bottom. The newly added rows and columns are drawn in red in the diagram below. Label the added rows as $\{\overline 1, \overline{2}, \dots , \overline{n-r}\}$ from bottom to top and label the added columns as $\{\overline{1}, \overline{2}, \dots , \overline{m-r}\}$ from left to right, as shown in the diagram.
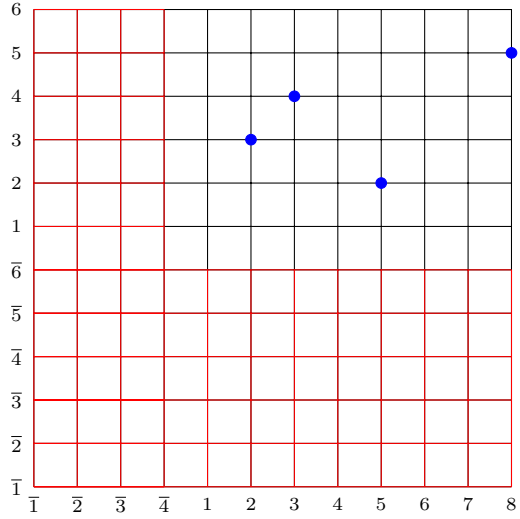
\begin{center}
    \begin{tikzpicture}[x=1.5em,y=1.5em]
    \draw[step=1,black,thin] (-4,-6) grid (7,5);
    \filldraw [blue] (1,2) circle (2pt);
    \filldraw [blue] (2,3) circle (2pt);
    \filldraw [blue] (4,1) circle (2pt);
    \filldraw [blue] (7,4) circle (2pt);
    \draw[red] (-4,-6)--(7,-6)
    (-4,-5)--(7,-5)
    (-4,-4)--(7,-4)
    (-4,-3)--(7,-3)
    (-4,-2)--(7,-2)
    (-4,-1)--(7,-1)
    (-4,5)--(-4,-6)
    (-3,5)--(-3,-6)
    (-2,5)--(-2,-6)
    (-1,5)--(-1,-6)
    (-4,5)--(-1,5)
    (-4,4)--(-1,4)
    (-4,3)--(-1,3)
    (-4,2)--(-1,2)
    (-4,1)--(-1,1)
    (-4,0)--(-1,0)
    (-4,-1)--(-1,-1)
    (-4,-2)--(-1,-2)
    (-4,-3)--(-1,-3)
    (-4,-4)--(-1,-4)
    (-4,-5)--(-1,-5)
    (-4,-6)--(-1,-6)
    (0,-1)--(0,-6)
    (1,-1)--(1,-6)
    (2,-1)--(2,-6)
    (3,-1)--(3,-6)
    (4,-1)--(4,-6)
    (5,-1)--(5,-6)
    (6,-1)--(6,-6)
    (7,-1)--(7,-6);
    \node [black] at (-4.4,-6) {\tiny $\overline{1}$};
    \node [black] at (-4.4,-5) {\tiny $\overline{2}$};
    \node [black] at (-4.4,-4) {\tiny $\overline{3}$};
    \node [black] at (-4.4,-3) {\tiny $\overline{4}$};
    \node [black] at (-4.4,-2) {\tiny $\overline{5}$};
    \node [black] at (-4.4,-1) {\tiny $\overline{6}$};
    \node [black] at (-4,-6.4) {\tiny $\overline{1}$};
    \node [black] at (-3,-6.4) {\tiny $\overline{2}$};
    \node [black] at (-2,-6.4) {\tiny $\overline{3}$};
    \node [black] at (-1,-6.4) {\tiny $\overline{4}$};
    \node [black] at (0,-6.4) {\tiny $1$};
    \node [black] at (1,-6.4) {\tiny $2$};
    \node [black] at (2,-6.4) {\tiny $3$};
    \node [black] at (3,-6.4) {\tiny $4$};
    \node [black] at (4,-6.4) {\tiny $5$};
    \node [black] at (5,-6.4) {\tiny $6$};
    \node [black] at (6,-6.4) {\tiny $7$};
    \node [black] at (7,-6.4) {\tiny $8$};
    \node [black] at (-4.4,0) {\tiny $1$};
    \node [black] at (-4.4,1) {\tiny $2$};
    \node [black] at (-4.4,2) {\tiny $3$};
    \node [black] at (-4.4,3) {\tiny $4$};
    \node [black] at (-4.4,4) {\tiny $5$};
    \node [black] at (-4.4,5) {\tiny $6$};
    
    \end{tikzpicture} 
\end{center}

For $i \leq r'-r$, we put a point at $(\overline{i},\overline{i})$, as drawn in black in the diagram below. Finally, let $i_1 < i_2 < \dots < i_{n-r'}$ be the indices of columns that do not contain a rook, and let $j_1 < j_2 < \dots < j_{m-r'}$ be the indices of rows that do not contain a rook. For $1 \leq l \leq m-r'$, add a point at $(\overline{r'-r+l}, j_l)$; for $1 \leq k \leq n-r'$, add a point at $(i_k,\overline{r'-r+k})$. These added points are drawn in red in the diagram below.

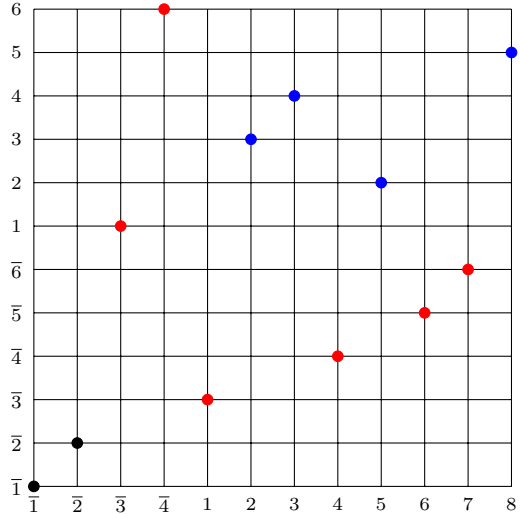
\begin{center}
    \begin{tikzpicture}[x=1.5em,y=1.5em]
    \draw[step=1,black,thin] (-4,-6) grid (7,5);
    \filldraw [blue] (1,2) circle (2pt);
    \filldraw [blue] (2,3) circle (2pt);
    \filldraw [blue] (4,1) circle (2pt);
    \filldraw [blue] (7,4) circle (2pt);
    \filldraw [black](-4,-6) circle (2pt);
    \filldraw [black](-3,-5) circle (2pt);
    \filldraw [red](-2,0) circle (2pt);
    \filldraw [red](-1,5) circle (2pt);    \filldraw [red](0,-4) circle (2pt);
    \filldraw [red](3,-3) circle (2pt);
    \filldraw [red](5,-2) circle (2pt);
    \filldraw [red](6,-1) circle (2pt);
    \node [black] at (-4.4,-6) {\tiny $\overline{1}$};
    \node [black] at (-4.4,-5) {\tiny $\overline{2}$};
    \node [black] at (-4.4,-4) {\tiny $\overline{3}$};
    \node [black] at (-4.4,-3) {\tiny $\overline{4}$};
    \node [black] at (-4.4,-2) {\tiny $\overline{5}$};
    \node [black] at (-4.4,-1) {\tiny $\overline{6}$};
    \node [black] at (-4,-6.4) {\tiny $\overline{1}$};
    \node [black] at (-3,-6.4) {\tiny $\overline{2}$};
    \node [black] at (-2,-6.4) {\tiny $\overline{3}$};
    \node [black] at (-1,-6.4) {\tiny $\overline{4}$};
    \node [black] at (0,-6.4) {\tiny $1$};
    \node [black] at (1,-6.4) {\tiny $2$};
    \node [black] at (2,-6.4) {\tiny $3$};
    \node [black] at (3,-6.4) {\tiny $4$};
    \node [black] at (4,-6.4) {\tiny $5$};
    \node [black] at (5,-6.4) {\tiny $6$};
    \node [black] at (6,-6.4) {\tiny $7$};
    \node [black] at (7,-6.4) {\tiny $8$};
    \node [black] at (-4.4,0) {\tiny $1$};
    \node [black] at (-4.4,1) {\tiny $2$};
    \node [black] at (-4.4,2) {\tiny $3$};
    \node [black] at (-4.4,3) {\tiny $4$};
    \node [black] at (-4.4,4) {\tiny $5$};
    \node [black] at (-4.4,5) {\tiny $6$};
    \end{tikzpicture} 
\end{center}

It can be easily checked that the diagram obtained this way has exactly one point in each row and each column, hence forming a permutation in $\symm_{n+m-r}$. We write $\ex(\RRR)$ for this {\em extended permutation} arising from $\RRR$, identified with the diagram on the extended board. 
\begin{remark}\label{rmk:extended-increasing}
    There is another way to define the extended permutation $\ex(\RRR)$. Define an ordering of row and column indices by
    \begin{itemize}
        \item if $i<j$ then $\overline{i} < \overline{j}$
        \item $\overline{i} < j$ for all $i,j$.
    \end{itemize}
    $\ex(\RRR)$ is the unique extension of $\RRR$ to a permutation in the $(n+m-r) \times (n+m-r)$ board such that the vertical coordinates of the points in the first $m-r$ columns form an increasing pattern and the horizontal coordinates of the points in the first $n-r$ rows form an increasing pattern.
\end{remark}

 Apply Viennot's shadow line construction to $\ex(R)$. From the construction (and by Remark~\ref{rmk:extended-increasing}), we have that the shadow lines $L_1,L_2,\dots,L_{n-r}$ will have infinite horizontal coordinates $\overline{1}, \overline{2}, \dots, \overline{n-r}$, and the shadow lines $L_1, L_2, \dots, L_{m-r}$ will have infinite vertical coordinates $\overline{1}, \overline{2}, \dots , \overline{m-r}$. Let $\eshadow(\RRR)$ be the shadow set of $\ex(\RRR)$, we call it the {\em extended shadow set} of $\RRR$. See the diagram for an example.

 
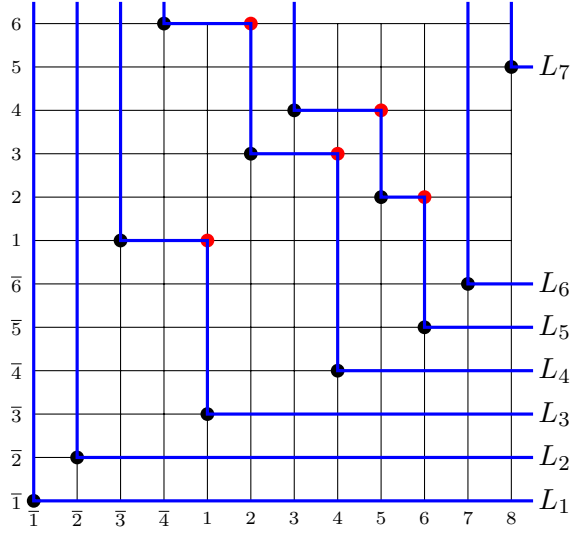
\begin{center}
    \begin{tikzpicture}[x=1.5em,y=1.5em, very thick,color = blue]
    \draw[step=1,black,thin] (-4,-6) grid (7,5);
    \filldraw [black] (1,2) circle (2pt);
    \filldraw [black] (2,3) circle (2pt);
    \filldraw [black] (4,1) circle (2pt);
    \filldraw [black] (7,4) circle (2pt);
    \filldraw [black](-4,-6) circle (2pt);
    \filldraw [black](-3,-5) circle (2pt);
    \filldraw [black](-2,0) circle (2pt);
    \filldraw [black](-1,5) circle (2pt);
    \filldraw [black](0,-4) circle (2pt);
    \filldraw [black](3,-3) circle (2pt);
    \filldraw [black](5,-2) circle (2pt);
    \filldraw [black](6,-1) circle (2pt);
    \filldraw [red](0,0) circle (2pt);
    \filldraw [red](1,5) circle (2pt);
    \filldraw [red](3,2) circle (2pt);
    \filldraw [red](4,3) circle (2pt);
    \filldraw [red](5,1) circle (2pt);
    \draw  (-4,5.5)--(-4,-6)--(7.5,-6)
    (-3,5.5)--(-3,-5)--(7.5,-5)
    (-2,5.5)--(-2,0)--(0,0)--(0,-4)--(7.5,-4)
    (-1,5.5)--(-1,5)--(1,5)--(1,2)--(3,2)--(3,-3)--(7.5,-3)
    (2,5.5)--(2,3)--(4,3)--(4,1)--(5,1)--(5,-2)--(7.5,-2)
    (6,5.5)--(6,-1)--(7.5,-1)
    (7,5.5)--(7,4)--(7.5,4);
    \node [black] at (-4.4,-6) {\tiny $\overline{1}$};
    \node [black] at (-4.4,-5) {\tiny $\overline{2}$};
    \node [black] at (-4.4,-4) {\tiny $\overline{3}$};
    \node [black] at (-4.4,-3) {\tiny $\overline{4}$};
    \node [black] at (-4.4,-2) {\tiny $\overline{5}$};
    \node [black] at (-4.4,-1) {\tiny $\overline{6}$};
    \node [black] at (-4,-6.4) {\tiny $\overline{1}$};
    \node [black] at (-3,-6.4) {\tiny $\overline{2}$};
    \node [black] at (-2,-6.4) {\tiny $\overline{3}$};
    \node [black] at (-1,-6.4) {\tiny $\overline{4}$};
    \node [black] at (0,-6.4) {\tiny $1$};
    \node [black] at (1,-6.4) {\tiny $2$};
    \node [black] at (2,-6.4) {\tiny $3$};
    \node [black] at (3,-6.4) {\tiny $4$};
    \node [black] at (4,-6.4) {\tiny $5$};
    \node [black] at (5,-6.4) {\tiny $6$};
    \node [black] at (6,-6.4) {\tiny $7$};
    \node [black] at (7,-6.4) {\tiny $8$};
    \node [black] at (-4.4,0) {\tiny $1$};
    \node [black] at (-4.4,1) {\tiny $2$};
    \node [black] at (-4.4,2) {\tiny $3$};
    \node [black] at (-4.4,3) {\tiny $4$};
    \node [black] at (-4.4,4) {\tiny $5$};
    \node [black] at (-4.4,5) {\tiny $6$};
    \node [black] at (8,-6) {$L_1$};
    \node [black] at (8,-5) {$L_2$};
    \node [black] at (8,-4) {$L_3$};
    \node [black] at (8,-3) {$L_4$};
    \node [black] at (8,-2) {$L_5$};
    \node [black] at (8,-1) {$L_6$};
    \node [black] at (8,4) {$L_7$};
    \end{tikzpicture} 
\end{center}
The following lemma is an important fact about $\eshadow(\RRR)$.
\begin{lemma}\label{lem:extended-shadow-upper-right}
    Let $\RRR \in \UZ_{n,m,r}$, we have that $\eshadow(\RRR)$ is also a rook placement in $[n] \times [m]$. That is, the shadow set of $\ex(\RRR)$ is contained in the $n \times m$ box in the upper right corner of the diagram, with column indices in $[n]$ and row indices in $[m]$.
\end{lemma}

\begin{proof}
    Let $(P,Q)$ be the pair of tableaux corresponding to $\ex(\RRR)$ in the Schensted correspondence. We adopt the notion in Remark~\ref{rmk:extended-increasing} so the content of $P$ is $\{\overline{1},\overline{2}, \dots,\overline{n-r}, 1,2,\dots m\}$ and the content of $Q$ is $\{ \overline{1}, \overline{2}, \dots, \overline{m-1}, 1,2, \dots n\}$. According to Viennot's construction in ~\cite{Viennot}, we have that $\overline{1},\overline{2},\dots,\overline{n-r}$ are in the first row of $P$ as they are the coordinates of the infinite horizontal rays of shadow lines $L_1,L_2,\dots,L_{n-r}$. Similarly, $\overline{1},\overline{2},\dots,\overline{m-r}$ are in the first row of $Q$. As the horizontal coordinates of points in the shadow set are entries in $P$ beyond the first row, and the vertical coordinates of points in the shadow set are entries in $Q$ beyond the shadow set, we conclude that $\eshadow(\RRR) = \SSS(\ex(\RRR)) \subseteq [n] \times [m]$.
\end{proof}

Given a rook placement $\RRR \subseteq [n] \times [m]$, the following proposition generalizes~\ref{lem: shadow-set-dertermine}. It characterizes whether $\RRR$ is the extended shadow set of some other rook placement $\tilde{\RRR}$ with size at least $r$.

\begin{proposition}\label{prop:equivalent-description-of-shadow-set}
    Let $\RRR$ be a rook placement on the $n\times m$ board and apply the Viennot shadow construction to $\RRR$, yielding shadow lines $L_1,\dots,L_p$. Define two sequences $(x_i)_{i=1}^n$ and $(y_j)_{j=1}^m$ over the alphabet $\{-1,0,1\}$ by
    \[x_i=\begin{cases}
        1, &\text{if one of the shadow lines $L_1,\dots,L_p$ has a vertical ray at $x=i$,} \\
        -1, &\text{if the vertical line $x=i$ does not meet $\RRR$,} \\
        0, &\text{otherwise.}
    \end{cases}\]
    and
    \[y_j=\begin{cases}
        1, &\text{if one of the shadow lines $L_1,\dots,L_p$ has a horizontal ray at $y=j$,} \\
        -1, &\text{if the horizontal line $y=j$ does not meet $\RRR$,} \\
        0, &\text{otherwise.}
    \end{cases}\]
    Then $\RRR=\eshadow(\tilde{\RRR})$ is the extended shadow set of some rook placement $\tilde{\RRR}\in\UZ_{n,m,r}$ if and only if for all $1\le i\le n$ and all $1\le j\le m$ we have $x_1+x_2+\dots+x_i\le m-r$ and $y_1+y_2+\dots+y_j\le n-r$.
\end{proposition}
\begin{proof}
    Suppose $\RRR = \eshadow(\tilde{\RRR})$ for some $\tilde{\RRR} \in \UZ_{n,m,r}$. Let $\ex(\tilde{\RRR})$ be the extended permutation of $\tilde{\RRR}$. Using the same notion as in the proof of Lemma~\ref{lem:extended-shadow-upper-right}, if $\ex(\tilde{\RRR}) \mapsto (P,Q)$ under the Schensted correspondence, the horizontal rays of $L_1,L_2,\dots,L_p$ give the second row of $P$ and the vertical rays of $L_1,L_2,\dots,L_p$ give the second row of $Q$. The $y$-coordinates that do not appear in $\RRR$ together with $\overline{1},\overline{2},\dots,\overline{n-r}$ form the first row of $P$, where the first $n-r$ entries are $\overline{1},\overline{2},\dots,\overline{n-r}$. Similarly, the $x$-coordinates that do not appear in $\RRR$ together with $\overline{1},\overline{2}, \dots, \overline{m-r}$ form the first row of $Q$, and the first $m-r$ entries are $\overline{1},\overline{2}, \dots, \overline{m-r}$. Since $P$ and $Q$ are standard, we must have all prefix sums of the sequence $x_1x_2 \dots x_n$ are at most $m-r$ and all prefix sums of the sequence $y_1y_2 \dots y_m$ are at most $n-r$.

    Now assume that all prefix sums of $x_1 x_2 \dots x_n$ are at most $m-r$ and all prefix sums of $y_1y_2 \dots y_m$ are at most $n-r$. Embed $\RRR$ in top right corner of the $(n+m-r) \times (n+m-r)$ board, and apply Viennot's shadow construction to the set $R$ to get a pair $(P',Q')$ of partial standard Young tableaux where entries of $P'$ are the $x$ coordinates of points in $\RRR$ and entries in $Q'$ are the $y$ coordinates of points in $\RRR$. Add a first row consisting of entries $\overline{1},\overline{2},\dots,\overline{n-r}$ together with $y$ coordinates not appearing in $\RRR$ to $P'$, and add a first row consisting of entries $\overline{1},\overline{2},\dots,\overline{n-r}$ together with $x$ coordinates not appearing in $\RRR$ to $Q'$. The pair of tableaux $(P,Q)$ obtained this way is standard by the conditions on prefix sums. Consider the diagram of the permutation $w \in \symm_{n+m-r}$ on the $(n+m-r) \times (n+m-r)$ board. From our construction of the first row of $P$ and $Q$, it must be the case that the vertical coordinates of the points in the first $m-r$ columns form an increasing pattern and the horizontal coordinates of the points in the first $n-r$ rows form an increasing pattern. Hence, let $\tilde{\RRR}$ be the rook placement in the upper-right $n \times m$ board in the diagram of $w$, it must be the case that $w = \ex(\tilde{\RRR})$ and that $\RRR = \eshadow (\tilde{\RRR})$ by Remark~\ref{rmk:extended-increasing}.
\end{proof}

\subsection{Extended shadow monomials and spanning}\label{subsec:shadow-mono-span}
Our next goal is to obtain a basis for $\CC[\xxx_{n\times m}]/I_{n,m,r}$ using the combinatorics of the previous section.

\begin{lemma}\label{lem:rook-span}
    The family of monomials $\{\mmm(\RRR)\,:\,\RRR\in\bigsqcup_{d=0}^{\min\{m,n\}}\ZZZ_{n,m,d}\}$ descends to a spanning set of $\CC[\xxx_{n\times m}]/I_{n,m,r}$.
\end{lemma}
\begin{proof}
    This follows immediately since $I_{n,m,r}$ contains all the products $x_{i,j}\cdot x_{i,j^\prime}$ and $x_{i,j}\cdot x_{i^\prime,j}$ of variables in the same row or column.
\end{proof}

However, the spanning set given by Lemma~\ref{lem:rook-span} is far from a basis. To extract a basis from this spanning set, we introduce a class of special monomial orders, namely the \emph{diagonal monomial orders}. Recall that the \emph{lexicographical order} in an ordered set of variables $y_1>y_2>\dots>y_N$ is given by $y_1^{a_1}\dots y_N^{a_N} < y_1^{b_1}\dots y_N^{b_N}$ if and only if there exists an integer $1\le M\le N$ such that $a_i=b_i$ for $1\le i <M$ and $a_M<b_M$.

\begin{definition}\label{def:diag-order}
    Observe that any total order $<$ on the variable set $\xxx_{n\times m}$ induces a lexicographical order on the monomials in $\CC[\xxx_{n \times m}]$. Such a monomial order is called {\em diagonal} if the total order $<$ satisfies $x_{i,j} < x_{i^\prime,j^\prime}$ whenever $i+j<i^\prime+j^\prime$.
\end{definition}

It is easy to construct a diagonal monomial order. For instance, the variable order
\[x_{1,1}>x_{2,1}>x_{1,2}>x_{3,1}>x_{2,2}>x_{1,3}>\dots>x_{n,m-1}>x_{n-1,m}>x_{n,m}\]
induces a diagonal monomial order.

From now on, we fix a diagonal monomial order $\diag$ on the set of monomials in $\CC[\xxx_{n\times m}]$.

\begin{definition}\label{def:shadow-mono}
    Let $\RRR\in\UZ_{n,m,r}$. The \emph{extended shadow monomial} $\es(\RRR)\in\CC[\xxx_{n\times m}]$ is defined as $\es(\RRR)\coloneqq\mmm(\eshadow(\RRR))$.
\end{definition}

We now prove the last technical lemma to obtain a smaller spanning set. 
\begin{lemma}\label{lem:linear-relation-induction}
    For $0\le d\le\min\{m,n\}$ and a rook placement $\RRR\in\ZZZ_{n,m,d}$ that is not the extended shadow set of any rook placement in $\UZ_{n,m,r}$, we have
    \[\mmm(\RRR)\in\spa\{\mmm(\RRR^{\prime})\,:\,\RRR^{\prime}\in\ZZZ_{n,m,d},\,\mmm(\RRR^{\prime})\diag\mmm(\RRR)\} + I_{n,m,r}.\]
\end{lemma}
\begin{proof}
Applying Viennot's shadow construction to $\RRR$, we obtain shadow lines $L_1,\dots,L_p$, together with sequences $(x_i)_{i=1}^n$ and $(y_j)_{j=1}^m$ as in Proposition~\ref{prop:equivalent-description-of-shadow-set}. Because $\RRR$ is not the extended shadow set of any rook placements in $\UZ_{n,m,r}$, Proposition~\ref{prop:equivalent-description-of-shadow-set} implies that either $\sum_{i=1}^k x_i > m-r$ for some $1\le k\le n$ or $\sum_{j=1}^l y_j >n-r$ for some $1\le l\le m$. 

Without loss of generality, we assume that $\sum_{i=1}^k x_i >m-r$ for some $1\le k\le n$ and further require $k$ to be the smallest integer with this property. Then $x_k>0$ and thus $x_k=1$, implying that some shadow line $L_t$ has a vertical ray at $x=k$. We define a size $t$ subset $\RRR_1 = \{(i_1,j_1),\dots,(i_t,j_t)\}\subseteq\RRR$ as follows. Start at the vertical ray of $L_t$ and go south. Let $(i_t,j_t)$ be the first rook of $\RRR$ encountered (in particular, we have $i_t=k$). Then go west from $(i_t,j_t)$ until we encounter a vertical segment of the shadow line $L_{t-1}$ and turn left. Go south until we encounter a rook $(i_{t-1},j_{t-1})\in\RRR$. Repeating this process, we obtain a sequence of rooks $(i_t,j_t),(i_{t-1},j_{t-1}),\dots,(i_1,j_1)$ sequentially, such that $(i_u,j_u)$ is on the shadow line $L_u$ ($u=1,\dots,t$), with $i_1<i_2<\dots<i_t$, and $j_1<j_2<\dots<j_t$. Let $\RRR_2=\RRR\setminus\RRR_1$. 

For example, let $n=m=8$ and $r=7$. Let $\RRR=\{(2,2),(3,5),(4,3),(5,6),(6,1),(7,7)\}\in\ZZZ_{8,8,6}$ be a rook placement on the $8\times 8$ board. We mark all rooks of $\RRR$ using red points and draw black shadow lines $L_1,\dots,L_4$. The sequence $(x_i)_{i=1}^8$ is $-1,1,1,0,1,0,1,-1$ where the smallest $k$ such that $\sum_{i=1}^k x_i>m-r$ is $k=5$. Since $x=5$ corresponds to the vertical ray of the shadow line $L_3$, we start from this vertical ray and go south, then west, and so on, yielding $\{(i_1,j_1),(i_2,j_2),(i_3,j_3)\}$. Our track is shown in blue as follows.

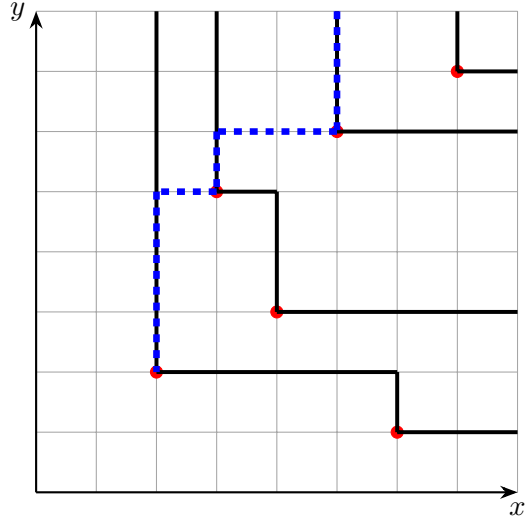
\begin{center}
\begin{tikzpicture}[scale=0.8]
    \draw[gray!70, thin] (0,0) grid (8,8);
    \draw[thick, -Stealth] (0,0) -- (8,0) node[below] {$x$};
    \draw[thick, -Stealth] (0,0) -- (0,8) node[left]  {$y$};
    \filldraw[red] (2,2) circle (.1)
    (3,5) circle (.1)
    (4,3) circle (.1)
    (5,6) circle (.1)
    (6,1) circle (.1)
    (7,7) circle (.1);
    \draw[line width = 1.5pt] (2,2) -- (2,8)
    (2,2) -- (6,2)
    (6,2) -- (6,1)
    (6,1) -- (8,1)
    (3,5) -- (3,8)
    (3,5) -- (4,5)
    (4,5) -- (4,3)
    (4,3) -- (8,3)
    (5,6) -- (5,8)
    (5,6) -- (8,6)
    (7,7) -- (7,8)
    (7,7) -- (8,7);
    \draw[dashed, line width = 2.5pt, blue] (5,8) -- (5,6)
    (5,6) -- (3,6)
    (3,6) -- (3,5)
    (3,5) -- (2,5)
    (2,5) -- (2,2);
\end{tikzpicture}
\end{center}

Consequently we obtain $(i_1,j_1) = (2,2)$, $(i_2,j_2) = (3,5)$, and $(i_3,j_3) = (5,6)$, so \[\RRR_1=\{(2,2),(3,5),(5,6)\}\] and \[\RRR_2 = \RRR\setminus\{(2,2),(3,5),(5,6)\} = \{(4,3),(6,1),(7,7)\}\] in the example above. 

Back to our proof. Define $S\subseteq[n]$ by
\[S = \{i_1,i_2,\dots,i_t\}\sqcup\{i\in\{i_t+1,\dots,n\}\,:\,\text{the vertical line $x=i$ does not meet $\RRR$}\}.\]
\textbf{Claim:} $|S|>n+m-d-r$.

In fact, $\sum_{i=1}^k x_i > m-r$ implies that 
\[t>|\{i\in[i_t]\,:\, \text{the vertical line $x=i$ does not meet $\RRR$}\}| + m-r\]
and hence
\begin{align*}
   |S|= & t + |\{i\in\{i_t+1,\dots,n\}\,:\,\text{the vertical line $x=i$ does not meet $\RRR$}\}|  \\
   > & |\{i\in[i_t]\,:\, \text{the vertical line $x=i$ does not meet $\RRR$}\}| + m-r + \\ & |\{i\in\{i_t+1,\dots,n\}\,:\,\text{the vertical line $x=i$ does not meet $\RRR$}\}| \\
   = & |\{i\in[n]\,:\,\text{the vertical line $x=i$ does not meet $\RRR$}\}| +m-r \\
   = & n-d +m-r =n+m-d-r
\end{align*}
which completes the proof of the claim above.

Then Corollary~\ref{cor:ann-of-I_{n,m,r}-extend} reveals that
\begin{equation}\label{eq:linear-relation}
\sum_{w\in\symm_S}(w,1)\cdot\mmm(\RRR) \in I_{n,m,r}.
\end{equation}
Note that $\mmm(\RRR) = \mmm(\RRR_1)\cdot\mmm(R_2)$ where for all $w\in\symm_{S}$ the action of $(w,1)$ does not change $\mmm(\RRR_2)$. Therefore, we have that
\[\sum_{w\in\symm_S}(w,1)\cdot\mmm(\RRR) = \bigg(\sum_{w\in\symm_S}(w,1)\cdot\mmm(\RRR_1)\bigg)\cdot\mmm(\RRR_2)\]
where $\mmm(\RRR_1)$ is the initial monomial of $\sum_{w\in\symm_S}(w,1)\cdot\mmm(\RRR_1)$ under the monomial order $\diag$, indicating that $\mmm(\RRR_1)\cdot\mmm(\RRR_2) = \mmm(\RRR)$ is the initial monomial of $\sum_{w\in\symm_S}(w,1)\cdot\mmm(\RRR)$. As a result, the initial monomial of the expression on the left-hand side of Equation~\eqref{eq:linear-relation} is $\mmm(\RRR)$, finishing the proof.
\end{proof}

Now we are ready to provide a smaller spanning set than Lemma~\ref{lem:rook-span}. In the next section, we will see that this spanning set is actually the standard monomial basis of $\CC[\xxx_{n\times m}]/I_{n,m,r}$ with respect to $\diag$
\begin{lemma}\label{lem:shadow-span}
    The family of monomials $\{\es(\RRR)\,:\,\RRR\in\UZ_{n,m,r}\}$ descends to a spanning set of $\CC[\xxx_{n\times m}]/I_{n,m,r}$.
\end{lemma}
\begin{proof}
    By Lemma~\ref{lem:rook-span}, it suffices to show that for all the rook placements $\RRR$ on the $n\times m$ board we have that
    \begin{equation*}
        \mmm(\RRR)\in\spa\{\es(\RRR^\prime)\,:\,\RRR^\prime\in\UZ_{n,m,r}\} + I_{n,m,r}.
    \end{equation*}
    This is immediately from Lemma~\ref{lem:linear-relation-induction} which enables us to do induction with respect to the monomial order $\diag$ on $\{\mmm(\RRR)\,:\,\text{$\RRR$ is a rook placement on the $n\times m$ board}\}$.
\end{proof}
\begin{remark}\label{rmk:shadow-span}
    In fact, the proof of Lemma~\ref{lem:shadow-span} gives a slightly stronger result than Lemma~\ref{lem:shadow-span} itself: For each rook placement $\RRR$ on the $n\times m$ board, we have that
    \[\mmm(\RRR)\in\spa\{\es(\RRR^\prime)\,:\,\RRR^\prime\in\UZ_{n,m,r},\,\, \es(\RRR^\prime)\eqdiag\mmm(\RRR)\} + I_{n,m,r}.\]
\end{remark}

\subsection{Standard monomial basis and Hilbert series}\label{subsec:basis-hilb}

As promised, we strengthen Lemma~\ref{lem:shadow-span} to obtain the standard monomial basis of $\CC[\xxx_{n\times m}]/I_{n,m,r}$.
\begin{theorem}\label{thm:basis}
    The family of monomials $\{\es(\RRR)\,:\,\RRR\in\UZ_{n,m,r}\}$ descends to the standard monomial basis of $\CC[\xxx_{n\times m}]/I_{n,m,r}$. Furthermore, we have $\gr \, \II(\UZ_{n,m,r}) = I_{n,m,r}$.
\end{theorem}
\begin{remark}\label{rmk:thm:basis}
    Consider the case $r=0$. Recall that Proposition~\ref{prop:toy-basis} provides a basis $\{\mmm(\RRR)\,:\,\RRR\in\ZZZ_{n,m}\}$ which contains the standard monomial basis $\{\es(\RRR)\,:\,\RRR\in\UZ_{n,m,0} = \ZZZ_{n,m}\}$ given by Theorem~\ref{thm:basis}. Therefore, we have
    \[\{\mmm(\RRR)\,:\,\RRR\in\ZZZ_{n,m}\} = \{\es(\RRR)\,:\,\RRR\in\UZ_{n,m,0} = \ZZZ_{n,m}\},\]
    indicating that
    \[\ZZZ_{n,m} = \{\eshadow(\RRR)\,:\,\RRR\in\ZZZ_{n,m}\}.\]
    In fact, this fact is also immediate from Proposition~\ref{prop:equivalent-description-of-shadow-set}.
\end{remark}

In order to prove Theorem~\ref{thm:basis}, we need the following lemma on ideal containment.

\begin{lemma}\label{lem: ideal-containment}
    $I_{n,m,r}\subseteq\gr \, \II(\UZ_{n,m,r})$.
\end{lemma}
\begin{proof}
    For any rook placement $\RRR\in\UZ_{n,m,r}$, each column (resp. row) of $[n]\times[m]$ has at most one rook. Consequently, for all $1\le i\le n$ and $1\le j,j^\prime\le m$ we have $x_{i,j}\cdot x_{i,j^\prime}\in\II(\UZ_{n,m,r})$ and thus $x_{i,j}\cdot x_{i,j^\prime}\in\gr \, \II(\UZ_{n,m,r})$ (resp. for all $1\le i,i^\prime\le n$ and $1\le j\le m$ we have $x_{i,j}\cdot x_{i^\prime,j}\in\gr \, \II(\UZ_{n,m,r})$). Additionally, since $\RRR$ has at least $r$ rooks, any union of $n-r+1$ distinct columns of $[n]\times[m]$ must meet $\RRR$. Thus, for all $1\le i_1<\dots<i_{n-r+1}\le n$ we have
    \[\prod_{k=1}^{n-r+1}\Bigg(\sum_{j=1}^m x_{i_k,j} - 1\Bigg) \in \II(\UZ_{n,m,r}),\] which implies that
    \[\prod_{k=1}^{n-r+1}\Bigg(\sum_{j=1}^m x_{i_k,j}\Bigg) \in \gr \, \II(\UZ_{n,m,r}).\] Similarly, for all $1\le j_1<\dots<j_{m-r+1}\le m$ we have
    \[\prod_{k=1}^{m-r+1}\Bigg(\sum_{i=1}^n x_{i,j_k}\Bigg) \in \gr \, \II(\UZ_{n,m,r}).\]
    The proof is thus complete as all generators of $I_{n,m,r}$ are in $\gr \, \II(\UZ_{n,m,r})$.
\end{proof}
With the lemma, we now give the proof of Theorem~\ref{thm:basis}
\begin{proof}[Proof of Theorem~\ref{thm:basis}]
    Lemma~\ref{lem: ideal-containment}yields a surjective $\CC$-linear map
    \[\CC[\xxx_{n\times m}]/I_{n,m,r}\twoheadrightarrow R(\UZ_{n,m,r}),\] implying that $\dim_\CC(\CC[\xxx_{n\times m}]/I_{n,m,r})\ge\dim_\CC(R(\UZ_{n,m,r}))$. However, Lemma~\ref{lem:shadow-span} implies that $\dim_\CC(\CC[\xxx_{n\times m}]/I_{n,m,r})\le|\UZ_{n,m,r}|=\dim_\CC(R(\UZ_{n,m,r}))$. Therefore, we conclude that \[\dim_\CC(\CC[\xxx_{n\times m}]/I_{n,m,r})=\dim_\CC(R(\UZ_{n,m,r})),\] so the spanning set $\{\es(\RRR)\,:\,\RRR\in\UZ_{n,m,r}\}$ in Lemma~\ref{lem:shadow-span} is indeed a basis of $\CC[\xxx_{n\times m}]/I_{n,m,r}$. Furthermore, the surjection $\CC[\xxx_{n\times m}]/I_{n,m,r}\twoheadrightarrow R(\UZ_{n,m,r})$ above is actually a linear isomorphism, and hence \[I_{n,m,r} = \gr \, \II(\UZ_{n,m,r}).\]

    As we have shown that $\{\es(\RRR)\,:\,\RRR\in\UZ_{n,m,r}\}$ descends to a basis of $\CC[\xxx_{n\times m}]/I_{n,m,r}$, it remains to show that $\{\es(\RRR)\,:\,\RRR\in\UZ_{n,m,r}\}$ descends to the standard monomial basis of $\CC[\xxx_{n\times m}]/I_{n,m,r}$. Assume, for the sake of contradiction, that there exists some $\RRR\in\UZ_{n,m,r}$ such that $\es(\RRR)$ is the initial monomial of some $f\in I_{n,m,r}$. Then we have that
    \begin{align}\label{eq:contradic-basis}
        \es(\RRR) \equiv \sum_{\substack{\RRR^\prime\in\ZZZ_{n,m,|\RRR|}\\ \mmm(\RRR^\prime) < \es(\RRR)}} c_{\RRR^{\prime}}\cdot\mmm(\RRR^\prime) \mod{I_{n,m,r}}
    \end{align}
    where $c_{\RRR^\prime}\in\CC$. However, Remark~\ref{rmk:shadow-span} implies that the right-hand side of Equation~\ref{eq:contradic-basis} can be replaced with a linear combination of some extended shadow monomials \[\{\es(\RRR^\prime)\,:\,\RRR^\prime\in\UZ_{n,m,r},\,\,\es(\RRR^\prime)\diag\es(\RRR)\},\]
    which means that extended shadow monomials $\{\es(\RRR)\,:\,\RRR\in\UZ_{n,m,r}\}$ are linearly dependent in $\CC[\xxx_{n\times m}]/I_{n,m,r}$. This yields a contradiction because we have shown the family of monomials $\{\es(\RRR)\,:\,\RRR\in\UZ_{n,m,r}\}$ descends to a basis of $\CC[\xxx_{n\times m}]/I_{n,m,r}$. Therefore, the family of monomials $\{\es(\RRR)\,:\,\RRR\in\UZ_{n,m,r}\}$ descends to the standard monomial basis of $\CC[\xxx_{n\times m}]/I_{n,m,r}$.
\end{proof}

Theorem~\ref{thm:basis} immediately gives the Hilbert series of $R(\UZ_{n,m,r}) = \CC[\xxx_{n\times m}]/I_{n,m,r}$. Given a permutation $w\in\symm_{n+m-r}$, a \emph{increasing subsequence of length $k$} is a sequence $1\le i_1<\dots<i_k\le n+m-r$ of integers such that $w(i_1)<\dots<w(i_k)$. We write
\[\mathrm{lis}(w)\coloneqq \text{length of the longest increasing subsequence of $w$}.\]
\begin{corollary}\label{cor:hilb}
    The Hilbert series of $R(\UZ_{n,m,r})$ is 
    \begin{equation}
       \Hilb(R(\UZ_{n,m,r});q) = \sum_{\RRR \in \UZ_{n,m,r}} q^{n+m-r-\mathrm{lis}(\ex(\RRR))}.
    \end{equation}
\end{corollary}
\begin{proof}
    By Theorem~\ref{thm:basis}, it suffices to show that
    $\deg(\es(\RRR)) = n+m-r-\mathrm{lis}(\ex(\RRR))$ for $\RRR\in\UZ_{n,m,r}$, which is equivalent to
    \[\lvert\eshadow(w)\rvert = n+m-r-\mathrm{lis}(\ex(\RRR)).\]
    Recall that $\eshadow(\RRR) = \SSS(\ex(\RRR))$ is the shadow set of $\ex(\RRR)\in\symm_{n+m-r}$. Consequently, it remains to show that for all $w\in\symm_{n+m-r}$ we have
    \begin{align}\label{eq:shdow-lis}\lvert\SSS(w)\rvert = n+m-r-\mathrm{lis}(w).\end{align}
    Let $k$ be the number of shadow lines of $w$ that we use to construct $\SSS(w)$. Write $w\mapsto(P,Q)$ according to the Schensted correspondence, where $P,Q$ have the same shape $\lambda$. The shadow line construction in Subsection~\ref{subsec:Viennot-shadow} implies that $k=\lambda_1$. In addition, Schensted correspondence possesses the standard property \cite[Thm. 1]{Schensted_1961} $\lambda_1 = \mathrm{lis}(w)$, so we have $k=\mathrm{lis}(w)$ and thus Equation~\eqref{eq:shdow-lis} holds.
\end{proof}

Surprisingly, although $\UZ_{n,m,r}$ looks larger than $\symm_n$, our extension operation $\ex$ embeds $\UZ_{n,m,r}$ into $\symm_{n+m-r}$, so that we can substitute $n=N-a,m=N-b,r=N-a-b$ and hence embed $\UZ_{n,m,r}$ into $\symm_N$. Therefore, while \cite[Corollary 3.13]{rhoades2024increasing} states that
\begin{align}\label{eq:rhoades-hilb}\Hilb(\CC[\xxx_{N\times N}]/I_{N,N,N};q) = \sum_{w\in\symm_N} q^{N-\mathrm{lis}(w)},\end{align}
the following modified version of Corollary~\ref{cor:hilb} refines the index set of Equation~\eqref{eq:rhoades-hilb}.
\begin{corollary}\label{cor:hilb-new}
    For $N\in\ZZ_{>0}$ and $0\le a,b<N$, we have
    \[\Hilb(\CC[\xxx_{(N-a)\times(N-b)}]/I_{N-a,N-b,N-a-b};q) = \sum_w q^{N-\mathrm{lis}(w)}\]
    where $w$ ranges over
    $\{w\in\symm_{N}\,:\,\text{$w(1)<w(2)<\dots<w(a)$ and $w^{-1}(1)<w^{-1}(2)<\dots<w^{-1}(b)$}\}$.
\end{corollary}
\begin{proof}
    Write $\mathrm{Inc\coloneqq}\{w\in\symm_{N}\,:\,\text{$w(1)<w(2)<\dots<w(a)$ and $w^{-1}(1)<w^{-1}(2)<\dots<w^{-1}(b)$}\}$ and let $n=N-a,m=N-b,r=N-a-b$, then we have
    \[\mathrm{Inc}=\{w\in\symm_{n+m-r}\,:\,\text{$w(1)<w(2)<\dots<w(m-r)$ and $w^{-1}(1)<w^{-1}(2)<\dots<w^{-1}(n-r)$}\}.\]
    According to Corollary~\ref{cor:hilb}, it suffices to show that
    \begin{align*}
        \ex\,:\,\UZ_{n,m,r}&\longrightarrow\mathrm{Inc} \\
        \RRR&\longmapsto\ex(\RRR)
    \end{align*}
    is a bijection.

    Since the diagram of $\RRR$ is contained in the diagram of $\ex(\RRR)$, it immediately follows that $\ex$ is injective. It remains to show that $\ex$ is surjective. For any $w\in\mathrm{Inc}$, choose maximal indices $1\le i_0\le m-r,1\le j_0\le n-r$ such that $w(i_0)\le n-r$ and $w^{-1}(j_0)\le m-r$. If such $i_0$ (resp. $j_0$) does not exist, let $i_0=0$ (resp. $j_0=0$). Then, we have $[i_0]\cap w^{-1}(\{j_0+1,\dots,n-r\}) = \varnothing$ and hence $w(i_0)\le j_0$. Similarly, we have $[j_0]\cap w(\{i_0+1,\dots,m-r\}) =\varnothing$ and hence $j_0\le w(i_0)$. These two inequalities force $w(i_0)=j_0$. Therefore, $w(1)<w(2)<\dots<w(i_0)$ indicates that $i_0\le j_0$, while $w^{-1}(1)<w^{-1}(2)<\dots<w^{-1}(j_0)$ indicates that $j_0\le i_0$. As a result, we have $i_0=j_0$ and $w(i)=i$ for all $i\in[i_0]$, so $w=\ex(\RRR)$ where $\RRR\in\UZ_{n,m,r}$ arises from the removal of the first $m-r$ columns and the first $n-r$ rows from the diagram of $w$.
\end{proof}




\section{Module structure}\label{sec:module}

We first provide some combinatorial results about symmetric functions. For convenience, we write $\langle s_{\lambda^{(1)}}\otimes s_{\lambda^{(2)}}\rangle F$ for the coefficient of $s_{\lambda^{(1)}}\otimes s_{\lambda^{(2)}}$ in the Schur expansion of any $F\in \Lambda\otimes_{\CC(q)}\Lambda$. Similarly, we write $\langle q^i\rangle f$ for the coefficient of $q^i$ in any generating function $f\in\CC[[q]]$.

\begin{lemma}\label{lem:coef}
    For $d,a,b,p,q\in\ZZ_{\ge 0}$, consider
    \[F := \sum_{\mu\vdash d}\{s_\mu\cdot h_a\}_{\lambda_1 =p}\otimes\{s_\mu\cdot h_b\}_{\lambda_1 = q}\in\Lambda\otimes_{\CC(q)}\Lambda\]
    and let $\lambda^{(1)}\vdash a+d$ and $\lambda^{(2)}\vdash b+d$ be partitions.
    
    \noindent If all the following conditions hold:
    \begin{itemize}
        \item $\lambda_1^{(1)} = p$,
        \item $\lambda_1^{(2)} = q$, and
        \item for all $i\ge 1$ we have $\min\{\lambda_i^{(1)},\lambda_i^{(2)}\}\ge\max\{\lambda_{i+1}^{(1)},\lambda_{i+1}^{(2)}\}$.
    \end{itemize}
    Then we have
    \[\langle s_{\lambda^{(1)}}\otimes s_{\lambda^{(2)}} \rangle F = 
        \Big\langle q^{d-\sum_{i=1}^\infty \max\{\lambda_{i+1}^{(1)},\lambda_{i+1}^{(2)}\}}\Big\rangle\Bigg(\prod_{i=1}^\infty\Bigg(\sum_{j=0}^{\min\{\lambda_i^{(1)},\lambda_i^{(2)}\} - \max\{\lambda_{i+1}^{(1)},\lambda_{i+1}^{(2)}\}}q^j\Bigg)\Bigg).
    \]
    Otherwise, we have
    \[\langle s_{\lambda^{(1)}}\otimes s_{\lambda^{(2)}} \rangle F = 
        0.
    \]
\end{lemma}
\begin{proof}
    If $s_{\lambda_1^{(1)}}\otimes s_{s_{\lambda_1^{(2)}}}$ appears in the Schur expansion of $F$, we must have $\lambda_1^{(1)}=p$ and $\lambda_1^{(2)}=q$ by the definition of $F$. Additionally, we must have $\min\{\lambda_i^{(1)},\lambda_i^{(2)}\}\ge\max\{\lambda_{i+1}^{(1)},\lambda_{i+1}^{(2)}\}$ for all $i\ge 1$, since both $\lambda^{(1)}$ and $\lambda^{(2)}$ arise from attaching a horizontal strip to the same partition $\mu$. In order to calculate $\langle s_{\lambda^{(1)}}\otimes s_{\lambda^{(2)}} \rangle F$, we only need to count all the partitions $\mu\vdash d$ such that: $\max\{\lambda_{i+1}^{(1)},\lambda_{i+1}^{(2)}\}\le\mu_i\le\min\{\lambda_i^{(1)},\lambda_i^{(2)}\}$ for all $i\ge 1$. Therefore, we have that
    \begin{align*}
        \langle s_{\lambda^{(1)}}\otimes s_{\lambda^{(2)}}\rangle F 
        &= \langle q^{d}\rangle\Bigg(\prod_{i=1}^\infty\Bigg(\sum_{j=\max\{\lambda_{i+1}^{(1)},\lambda_{i+1}^{(2)}\}}^{\min\{\lambda_i^{(1)},\lambda_i^{(2)}\}}q^j\Bigg)\Bigg) \\ 
        &=\langle q^{d-\sum_{i=1}^\infty \max\{\lambda_{i+1}^{(1)},\lambda_{i+1}^{(2)}\}}\rangle\Bigg(\prod_{i=1}^\infty\Bigg(\sum_{j=\max\{\lambda_{i+1}^{(1)},\lambda_{i+1}^{(2)}\}}^{\min\{\lambda_i^{(1)},\lambda_i^{(2)}\}}q^{j-\sum_{i=1}^\infty \max\{\lambda_{i+1}^{(1)},\lambda_{i+1}^{(2)}\}}\Bigg)\Bigg) \\ 
        &= \bigg\langle q^{d-\sum_{i=1}^\infty \max\{\lambda_{i+1}^{(1)},\lambda_{i+1}^{(2)}\}}\bigg\rangle\Bigg(\prod_{i=1}^\infty\Bigg(\sum_{j=0}^{\min\{\lambda_i^{(1)},\lambda_i^{(2)}\} - \max\{\lambda_{i+1}^{(1)},\lambda_{i+1}^{(2)}\}}q^j\Bigg)\Bigg).
    \end{align*}
\end{proof}

\begin{lemma}\label{lem:schur-interchange}
    Let $d,a,b,p,q\in\ZZ_{\ge 0}$ be nonnegative integers. We have
    \[\sum_{\mu\vdash d}\{s_\mu\cdot h_a\}_{\lambda_1 = p}\otimes\{s_\mu\cdot h_b\}_{\lambda_1 = q} = \sum_{\mu\vdash \big(d+a+b-\max\{p,q\}\big)}\{s_\mu\cdot h_{\max\{p,q\}-b}\}_{\lambda_1 = p}\otimes\{s_\mu\cdot h_{\max\{p,q\}-a}\}_{\lambda_1 = q}.\]
\end{lemma}
\begin{proof}
    We denote the left-hand side (resp. right-hand side) by $F$ (resp. $G$). It suffices to show that \[\langle s_{\lambda^{(1)}}\otimes s_{\lambda^{(2)}}\rangle F = \langle s_{\lambda^{(1)}}\otimes s_{\lambda^{(2)}}\rangle G\] for any pair of partitions $\lambda^{(1)}\vdash a+d$ and $\lambda^{(2)}\vdash b+d$.

    Consider three conditions:
    \begin{itemize}
        \item $\lambda_1^{(1)} = p$,
        \item $\lambda_1^{(2)} = q$, and
        \item for all $i\ge 1$ we have $\min\{\lambda_i^{(1)},\lambda_i^{(2)}\}\ge\max\{\lambda_{i+1}^{(1)},\lambda_{i+1}^{(2)}\}$. 
    \end{itemize}
    If one of them is not satisfied, Lemma~\ref{lem:coef} implies that
    $\langle s_{\lambda^{(1)}}\otimes s_{\lambda^{(2)}}\rangle F = 0 = \langle s_{\lambda^{(1)}}\otimes s_{\lambda^{(2)}}\rangle G$. We henceforth suppose that all these three conditions hold. Then Lemma~\ref{lem:coef} implies that
    \[\langle s_{\lambda^{(1)}}\otimes s_{\lambda^{(2)}} \rangle F = 
        \Big\langle q^{d-\sum_{i=1}^\infty \max\{\lambda_{i+1}^{(1)},\lambda_{i+1}^{(2)}\}}\Big\rangle f\]
        and
    \[\langle s_{\lambda^{(1)}}\otimes s_{\lambda^{(2)}} \rangle G = 
        \Big\langle q^{d+a+b-\max\{p,q\}-\sum_{i=1}^\infty \max\{\lambda_{i+1}^{(1)},\lambda_{i+1}^{(2)}\}}\Big\rangle f\]  
    where $f=\prod_{i=1}^\infty\Bigg(\sum_{j=0}^{\min\{\lambda_i^{(1)},\lambda_i^{(2)}\} - \max\{\lambda_{i+1}^{(1)},\lambda_{i+1}^{(2)}\}}q^j\Bigg)$. Therefore, it remains to show that
    \[\Big\langle q^{d-\sum_{i=1}^\infty \max\{\lambda_{i+1}^{(1)},\lambda_{i+1}^{(2)}\}}\Big\rangle f = \Big\langle q^{d+a+b-\max\{p,q\}-\sum_{i=1}^\infty \max\{\lambda_{i+1}^{(1)},\lambda_{i+1}^{(2)}\}}\Big\rangle f.\]
    Note that $f$ is actually a product of finitely many palindromic polynomials in $q$, $f$ itself is also a palindromic polynomial in $q$. Thus, it suffices to show that
    \[\Big(d-\sum_{i=1}^\infty \max\{\lambda_{i+1}^{(1)},\lambda_{i+1}^{(2)}\}\Big)+\Big(d+a+b-\max\{p,q\}-\sum_{i=1}^\infty \max\{\lambda_{i+1}^{(1)},\lambda_{i+1}^{(2)}\}\Big) = \deg(f),\]
    which is equivalent to
    \begin{align}\label{eq:schur-interchange}2d+a+b = \max\{p,q\} + \sum_{i=1}^\infty\min\{\lambda_i^{(1)},\lambda_i^{(2)}\} + \sum_{i=1}^\infty\max\{\lambda_{i+1}^{(1)},\lambda_{i+1}^{(2)}\}.\end{align}
    Since the right-hand side of Equation~\eqref{eq:schur-interchange} equals
    \begin{align*}
        &\max\{\lambda_1^{(1)},\lambda_1^{(2)}\} + \sum_{i=1}^\infty\min\{\lambda_i^{(1)},\lambda_i^{(2)}\} + \sum_{i=1}^\infty\max\{\lambda_{i+1}^{(1)},\lambda_{i+1}^{(2)}\} \\
        =&\sum_{i=1}^\infty\min\{\lambda_i^{(1)},\lambda_i^{(2)}\} + \sum_{i=1}^\infty\max\{\lambda_{i}^{(1)},\lambda_{i}^{(2)}\} \\
        =&\sum_{i=1}^\infty \Big(\min\{\lambda_i^{(1)},\lambda_i^{(2)}\} + \max\{\lambda_{i}^{(1)},\lambda_{i}^{(2)}\}\Big) \\
        =&\sum_{i=1}^\infty\Big(\lambda_i^{(1)} + \lambda_i^{(2)}\Big) = \lvert\lambda^{(1)}\rvert +\lvert\lambda^{(2)}\rvert = (a+d) +(b+d) =2d+a+b,
    \end{align*}
    which equals the left-hand side of Equation~\eqref{eq:schur-interchange}. Therefore, Equation~\eqref{eq:schur-interchange} holds, completing our proof.
\end{proof}

\begin{corollary}\label{cor:schur-sum-eq}
    \begin{align}\label{eq:schur-sum}\sum_{d=0}^{\min\{m,n\}}\Bigg\{\sum_{\mu\vdash d}(s_\mu\cdot h_{n-d})\otimes(s_\mu\cdot h_{m-d})\Bigg\}_{\lambda_1\le n+m-d-r} = \sum_{r^\prime =r}^{\min\{m,n\}}\sum_{\mu\vdash r^\prime}(s_\mu\cdot h_{n-r^\prime})\otimes(s_\mu\otimes h_{m-r^\prime})\end{align}
\end{corollary}
\begin{proof}
    Pieri's rule and Lemma~\ref{lem:schur-interchange} imply that the left-hand side equals
    \begin{align}\nonumber
        &\sum_{d=0}^{\min\{m,n\}}\sum_{\mu\vdash d}\{s_\mu\cdot h_{n-d}\}_{n-d\le\lambda_1\le n+m-d-r}\otimes\{s_\mu\cdot h_{m-d}\}_{m-d\le\lambda_1\le n+m-d-r} \\ \nonumber
        =&\sum_{\substack{(d,p,q)\in\ZZ^3\\0\le d\le \min\{m,n\}\\ n-d\le p\le n+m-d-r \\
        m-d\le q\le n+m-d-r}} \sum_{\mu\vdash d}\{s_\mu\cdot h_{n-d}\}_{\lambda_1 = p}\otimes\{s_\mu\cdot h_{m-d}\}_{\lambda_1 = q} \\ \nonumber
        =&\sum_{\substack{(d,p,q)\in\ZZ^3\\0\le d\le \min\{m,n\}\\ n-d\le p\le n+m-d-r \\ 
        m-d\le q\le n+m-d-r}} \sum_{\mu\vdash (m+n-d-\max\{p,q\})}\{s_\mu\cdot h_{\max\{p,q\}-m+d}\}_{\lambda_1 = p}\otimes\{s_\mu\cdot h_{\max\{p,q\}-n+d}\}_{\lambda_1 = q} \\
        \label{exp:sum}=&\sum_{r^\prime =r}^{\min\{m,n\}}\!\sum_{\substack{(d,p,q)\in\ZZ^3\\0\le d\le\min\{m,n\}\\p\ge n-d\\q\ge m-d\\n+m-d-\max\{p,q\}=r^\prime}} \!\sum_{\mu\vdash (m+n-d-\max\{p,q\})}\!\{s_\mu\cdot h_{\max\{p,q\}-m+d}\}_{\lambda_1 = p}\otimes\{s_\mu\cdot h_{\max\{p,q\}-n+d}\}_{\lambda_1 = q}
    \end{align}
    where the second equal sign arises from Lemma~\ref{lem:schur-interchange}, and the last equal sign arises from the following two facts:
    \begin{itemize}
    \item $p,q\le n+m-d-r$ if and only if $n+m-d-\max\{p,q\}\ge r$, and
    \item $p\ge n-d$ and $q\ge m-d$ together imply that $n+m-d-\max\{p,q\}\le\min\{m,n\}$.
    \end{itemize}
    Note that we can eliminate $d$ in the expression \eqref{exp:sum} using the condition $n+m-d-\max\{p,q\}=r^\prime$ under the summation, i.e., we can substitute $d=n+m-\max\{p,q\}-r^\prime$ in the expression \eqref{exp:sum}. Therefore, the left-hand side of Equation~\eqref{eq:schur-sum} equals
    \begin{align}\label{exp:no-d}
        &\sum_{r^\prime =r}^{\min\{m,n\}}\!\sum_{\substack{(p,q)\in\ZZ^2\\0\le n+m-\max\{p,q\}-r^\prime\le\min\{m,n\}\\p\ge \max\{p,q\}+r^\prime-m\\q\ge \max\{p,q\}+r^\prime-n\\}} \!\sum_{\mu\vdash r^\prime}\!\{s_\mu\cdot h_{n-r^\prime}\}_{\lambda_1 = p}\otimes\{s_\mu\cdot h_{m-r^\prime}\}_{\lambda_1 = q}.
    \end{align}
    \textbf{Claim}: All three inequalities under the second summation in \eqref{exp:no-d} can be discarded. 
    
    In fact, for any non-empty term $\{s_\mu\cdot h_{n-r^\prime}\}_{\lambda_1 = p}\otimes\{s_\mu\cdot h_{m-r^\prime}\}_{\lambda_1 = q}$ in \eqref{exp:no-d}, Pieri's rule yields that 
    \begin{equation}\label{ineq1}
    \max\{n-r^\prime,\mu_1\}\le p\le\mu_1+n-r^\prime\,\,\,(\le n)
    \end{equation} and 
    \begin{equation}\label{ineq2}
    \max\{m-r^\prime,\mu_1\}\le q\le\mu_1+m-r^\prime\,\,\,(\le m).
    \end{equation} Then 
    \begin{equation}\label{ineq3}
    \max\{m,n\}-r^\prime\le\max\{p,q\}\le\max\{m,n\}.
    \end{equation}
    Therefore, $0\le\min\{m,n\}-r^\prime=n+m-\max\{m,n\}-r^\prime\le n+m-\max\{p,q\}-r^\prime\le n+m-(\max\{m,n\}-r^\prime)-r^\prime=\min\{m,n\}$, so we have $0\le n+m-\max\{p,q\}-r^\prime\le\min\{m,n\}$ and hence the first equality under the summation in \eqref{exp:no-d} can be discard. Now we suppose that $p\le q$ without loss of generality. Then \eqref{ineq2} implies that $\max\{p,q\}+r^\prime-m=q+r^\prime-m\le\mu_1$, and \eqref{ineq1} implies that $p\ge\mu_1$. Consequently, we have $p\ge\max\{p,q\}+r^\prime-m$, so we can discard the second inequality under the summation of \eqref{exp:no-d}. Additionally, $\max\{p,q\}+r^\prime-n=q+r^\prime-n\le q$ because $r^\prime\le\min\{m,n\}$. Thus, we can discard the last inequality under the summation of \eqref{exp:no-d}. Therefore, the claim above holds.

    By this claim, the expression \eqref{exp:no-d} equals the right-hand side of Equation~\eqref{eq:schur-sum}, finishing the proof.
\end{proof}

Recall that we have studied the module structure of $\CC[\xxx_{n\times m}]/I_{n,m}$ in Section~\ref{subsec: r=0} and note that $I_{n,m}\subseteq I_{n,m,r}$. This observation reveals an $\symm_n\times\symm_m$-equivariant surjection $\CC[\xxx_{n\times m}]/I_{n,m}\twoheadrightarrow\CC[\xxx_{n\times m}]/I_{n,m,r}$. Thanks to this surjection, we finally obtain the graded $\symm_n\times\symm_n$-module structure of $R(\UZ_{n,m,r}) = \CC[\xxx_{n\times m}]/I_{n,m,r}$:
\begin{theorem}\label{thm:module-str}
    \[\grFrob(\CC[\xxx_{n\times m}]/I_{n,m,r};q) = \sum_{d=0}^{\min\{m,n\}}q^d\cdot\Bigg\{\sum_{\mu\vdash d}(s_\mu\cdot h_{n-d})\otimes(s_\mu\cdot h_{m-d})\Bigg\}_{\lambda_1\le n+m-d-r}\]
\end{theorem}
\begin{proof}
    The containment $I_{n,m}\subseteq I_{n,m,r}$ implies an $\symm_n\times\symm_m$-equivariant surjection
    \[\CC[\xxx_{n\times m}]/I_{n,m}\twoheadrightarrow\CC[\xxx_{n\times m}]/I_{n,m,r},\]
    which, together with Proposition~\ref{prop:toy-module-str}, means that
    \[\grFrob(\CC[\xxx_{n\times m}]/I_{n,m,r};q)\le\grFrob(\CC[\xxx_{n\times m}])/I_{n,m};q)=\sum_{d=0}^{\min\{m,n\}}q^d\cdot\sum_{\mu\vdash d}(s_\mu\cdot h_{n-d})\otimes(s_\mu\cdot h_{m-d}).\]
    Then Lemma~\ref{lem:ann-length} and Corollary~\ref{cor:ann-of-I_{n,m,r}-extend} implies that
    \begin{equation}\label{eq:graded-module-str-final}
        \grFrob(\CC[\xxx_{n\times m}]/I_{n,m,r};q)\le\sum_{d=0}^{\min\{m,n\}}q^d\cdot\Bigg\{\sum_{\mu\vdash d}(s_\mu\cdot h_{n-d})\otimes(s_\mu\cdot h_{m-d})\Bigg\}_{\lambda_1\le n+m-d-r}.
    \end{equation}
    However, Inequality~\eqref{eq:graded-module-str-final} implies that
    \begin{align}\label{eq:force-equal}
        \Frob(\CC[\xxx_{n\times m}]/I_{n,m,r})&\le\sum_{d=0}^{\min\{m,n\}}\Bigg\{\sum_{\mu\vdash d}(s_\mu\cdot h_{n-d})\otimes(s_\mu\cdot h_{m-d})\Bigg\}_{\lambda_1\le n+m-d-r} \\ \nonumber
        &=\sum_{r^\prime =r}^{\min\{m,n\}}\sum_{\mu\vdash r^\prime}(s_\mu\cdot h_{n-r^\prime})\otimes(s_\mu\otimes h_{m-r^\prime}) \\
        \nonumber
        &=\Frob\Bigg(\bigoplus_{r^\prime=r}^{\min\{m,n\}}\CC[\ZZZ_{n,m,r^\prime}]\Bigg) =\Frob(\CC[\UZ_{n,m,r}])   
    \end{align}
    where the first equal sign derives from Corollary~\ref{cor:schur-sum-eq}, and the second equal sign arises from Equation~\eqref{eq:graded-component-module-str=rook}. Now Theorem~\ref{thm:basis} tells us $R(\UZ_{n,m,r}) = \CC[\xxx_{n\times m}]/I_{n,m,r}$, so orbit harmonics yields an $\symm_n\times\symm_m$-module isomorphism 
    \[\CC[\xxx_{n\times m}]/I_{n,m,r} \cong \CC[\UZ_{n,m,r}],\]
    forcing all the inequalities in \eqref{eq:force-equal} to be equalities, and hence forcing Inequality~\eqref{eq:graded-module-str-final} to be the equality that we desire.
\end{proof}

\section{Conclusion}\label{sec:conclusion}
We provide an open problem regarding log-concavity. Recall that a sequence of positive real numbers $(a_1,a_2,\dots,a_N)$ is \emph{log-concave} if for all $1<i<N$ we have $a_i^2 \ge a_{i-1}\cdot a_{i+1}$. Chen conjectured \cite{chen2008logconcavityqlogconvexityconjectureslongest} that the sequence $(a_{n,1},a_{n,2},\dots,a_{n,n})$ is log-concave where $a_{n,k}$ is the number of permutations in $\symm_n$ with longest subsequence of length $k$. Using the terminology of orbit harmonics, Chen's conjecture means that the Hilbert series of the orbit harmonics ring $R(\symm_n)$ of the permutation matrix locus $\symm_n$ is log-concave.

We can generalize the definition of log-concavity by incorporating group action. Let $G$ be a group and $(V_1,V_2,\dots, V_N)$ is a sequence of $G$-modules. We say that $(V_1,V_2,\dots,V_N)$ is \emph{$G$-equivariant log-concave} if for all $1<i<N$ we have a $G$-module surjection $V_i\otimes V_i\twoheadrightarrow V_{i-1}\otimes V_{i+1}$ where the $G$-action is diagonal, i.e., $g\cdot(v\otimes w) = (g\cdot v)\otimes(g\cdot w)$. A graded $G$-module $V = \bigoplus_{d=0}^N V_d$ is \emph{$G$-equivariant log-concave} if the sequence $(V_1,V_2,\dots,V_N)$ is $G$-equivariant log-concave. Rhoades conjectured \cite{rhoades2024increasing} that $R(\symm_n)$ is $\symm_n\times\symm_n$-equivariant log-concave, which generalizes Chen's conjecture.

Our locus $\UZ_{n,m,r}$ coincides with the permutation matrix locus $\symm_n$ visited by Rhoades if $n=m=r$. Therefore, we raise the following conjecture, which further generalizes Rhoades's conjecture.
\begin{conjecture}\label{conj:log-concave}
    The graded $\symm_n\times\symm_m$-module $R(\UZ_{n,m,r})$ is $\symm_n\times\symm_m$-log-concave.
\end{conjecture}

Conjecture~\ref{conj:log-concave} has been checked by coding for $n\le 8,m\le 10$. The more flexibility provided by Conjecture~\ref{conj:log-concave} than Rhoades's conjecture, such as three parameters $n,m,r$ rather than $n$ itself, may permit more potential induction strategies and make the proof of log-concavity easier.

Another interesting future direction is to study the orbit harmonics ring $R(\ZZZ_{n,m,r})$. Since $\ZZZ_{n,m,r}\subseteq\UZ_{n,m,r}$, we need to add some elements to the generating set in Definition~\ref{def:ideal} to obtain the defining ideal $\gr \, \II(\ZZZ_{n,m,r})$, such as $\sum_{i=1}^n \sum_{j=1}^m x_{i,j}$ and $\mmm(\RRR)$ for $\RRR\in\UZ_{n,m,r+1}$. Some relevant results are presented in Zhu \cite{zhu2025rookplacementsorbitharmonics} focusing on the graded module structure.

Finally, we may also consider extending the results to colored rook placement loci. Let $N,k$ be positive integers, the \emph{$k$-colored permutation group} $\symm_{N,k}$ is the wreath product $(\ZZ/r\ZZ)\wr\symm_N$ which can be embedded into $\Mat_{N\times N}(\CC)$ by
\[\symm_{N,k} = \Big\{X\in\Mat_{N\times N}(\CC)\,:\,\begin{array}{cc}
     & \text{$X$ has exactly one nonzero entry in each row and column,} \\
     & \text{and nonzero entries of $X$ are $k$-th roots-of-unity}
\end{array}\Big\}.\]
Define the \emph{$k$-colored rook placement loci} $\ZZZ_{n,m,r}^{(k)}$ by
\[\ZZZ_{n,m,r}^{(k)}\coloneqq\Bigg\{X\in\Mat_{n\times m}(\CC)\,:\,\begin{array}{cc}
     & \text{$X$ has at most one nonzero entry in each row and column,} \\
     & \text{nonzero entries of $X$ are $k$-th roots-of-unity,} \\
     & \text{and $X$ possesses exactly $r$ nonzero entries}
\end{array}\Bigg\}\]
and let
\[\UZ_{n,m,r}^{(k)}\coloneqq\bigsqcup_{r^\prime =r}^{\min\{m,n\}}\ZZZ_{n,m,r^\prime}^{(k)}.\]
\begin{problem}\label{problem:colored}
    Find the graded $\symm_{n,k}\times\symm_{m,k}$-module structures of $R(\ZZZ_{n,m,r}^{(k)})$ and $R(\UZ_{n,m,r}^{(k)})$.
\end{problem}

\section{Acknowledgements}\label{sec:acknowledgements}
The authors are thankful to Brendon Rhoades for constructive suggestions about the content and structure of this paper.

\printbibliography

\end{document}